\documentclass[11pt]{amsart}

\usepackage[a-1b]{pdfx}

\usepackage{amssymb,amsmath,amsthm}%,mathtools,extpfeil}
\usepackage[all,cmtip]{xy}
\usepackage{graphicx}
\usepackage{fullpage,hyperref,enumerate}
\usepackage{enumitem}
\newtheorem{thm}{Theorem}[section]
\newtheorem{lem}[thm]{Lemma}

\theoremstyle{definition}
\newtheorem{defn}[thm]{Definition}
\newtheorem*{rmk}{Remark}

\newtheorem{question}{Question}
\numberwithin{thm}{section}
\numberwithin{equation}{section}

\DeclareMathOperator{\Vol}{Vol}
\DeclareMathOperator{\Area}{Area}

\DeclareMathOperator{\spt}{spt}
\DeclareMathOperator{\thin}{thin}
\DeclareMathOperator{\Int}{Int}
\DeclareMathOperator{\Index}{Index}
\DeclareMathOperator{\rel}{rel}
\DeclareMathOperator{\Tr}{Tr}
\DeclareMathOperator{\reg}{reg}
\DeclareMathOperator{\sided}{1-sided}

\DeclareMathOperator{\dist}{dist}
\DeclareMathOperator{\Lip}{Lip}
\DeclareMathOperator{\Ric}{Ric}
\DeclareMathOperator{\Hess}{Hess}
\DeclareMathOperator{\av}{av}
\DeclareMathOperator{\divergence}{div}

\begin{document}
\title{Generic Scarring for Minimal Hypersurfaces in Manifolds Thick at Infinity with a Thin Foliation at Infinity}
\author{Xingzhe Li}
\address{Department of Mathematics, Cornell University, Ithaca, NY 14853, USA}
\email{xl833@cornell.edu}

\begin{abstract}
    We show generic scarring phenomenon for minimal hypersurfaces in a class of complete non-compact manifolds. In particular, we prove that for any metric $g$ in a $C^{\infty}$-generic subset of the family of complete metrics which are thick at infinity with a thin foliation at infinity on a fixed $M^{n+1}$ of dimension $3 \leq (n + 1) \leq 7$, to any connected, closed, embedded, $2$-sided, stable minimal hypersurface $S \subset (M, g)$, there exists a sequence of closed, embedded, minimal hypersurfaces $\{\Sigma_{k}\}$ scarring along $S$, in the sense that the area of $\Sigma_{k}$ diverges to infinity, and when properly renormalized, $\Sigma_{k}$ converges to $S$ as varifolds.   
\end{abstract}
\maketitle

\section{Introduction}

The search for closed minimal hypersurfaces in compact manifolds has enjoyed significant progress in recent years. The Almgren-Pitts' min-max theory was much improved starting from the work of Marques-Neves \cite{Marques-Neves14, Marques-Neves16, Marques-Neves17, Marques-Neves18}. Subsequent exciting developments include the establishment of the Weyl law for area functional \cite{Liokumovich-Marques-Neves16}, the resolution of Yau’s Conjecture on the existence of infinitely many closed minimal hypersurfaces \cite{Song18}, the settlement of Multiplicity One Conjecture \cite{Zhou19} raised by Marques-Neves \cite{Marques-Neves16, Marques-Neves18} (see also \cite{Chodosh-Mantoulidis20}), and the proof of Four Spheres Theorem \cite{Wang-Zhou23-four-spheres}. About the same time, Li-Zhou developed the min-max theory for free boundary minimal hypersurfaces in the general Almgren-Pitts setting \cite{Li-Zhou16}. We have seen further advancement concerning free boundary minimal hypersurfaces inspired by results for closed minimal hypersurfaces \cite{Ambrozio-Carlotto-Sharp17, Guang-Li-Wang-Zhou19, Sun-Wang-Zhou20, Wang2021existence}. For minimal hypersurfaces in complete non-compact manifolds, there are comparably few results and most of them focus on the existence aspect \cite{Gro14, CHMR14, Chambers-Liokumovich16, CK18, Song19, Douglas23}.       

In contrast to the extensive theory on the existence of minimal hypersurfaces, there are few general results describing the possible spatial distributions of minimal hypersurfaces in an ambient manifold. In \cite{Marques-Neves-Song19}, Marques-Neves-Song proved that generically in dimensions $3$-$7$, there is a sequence of minimal hypersurfaces equidistributing on average. Inspired by \cite{Marques-Neves-Song19}, Li-Staffa \cite{LS23} verified the generic equidistribution of closed geodesics in dimension $2$ and embedded stationary geodesic nets in dimension $3$. On the other hand, Song-Zhou \cite{Song-Zhou21} showed that generically in dimensions $3$-$7$, scarring for minimal hypersurfaces happens as soon as there is a stable hypersurface. %Other results appeared in \cite{Lawson70, Colding-DeLellis05, Wiygul20}. 
It is important to note that for all results above, the ambient manifolds are compact. Then a natural question to ask is: does generic equidistribution or scarring happen in complete non-compact manifolds? In this paper, we partially answer this question by generalizing \cite[Theorem 0.1]{Song-Zhou21} to a class of complete manifolds. We show that for manifolds thick at infinity with a thin foliation at infinity, generically in dimensions $3$-$7$, scarring for minimal hypersurfaces happens as soon as there is a stable hypersurface. Denote by $\mathcal{F}_{\thin} \cap \Int(\mathcal{T}_{\infty})$ the family of complete metrics on $M$ thick at infinity with a thin foliation at infinity. %It is crucial to impose these conditions on the metric so that the geometry is controlled outside a compact region, and we will specify them later in this section. 
These conditions on the metric, being crucial to control the geometry outside a compact region, will be specified later in this section.

\begin{thm}[Main Theorem] \label{MAIN}
    Let $M^{n + 1}$ be a smooth manifold of dimension $3 \leq (n + 1) \leq 7$. For any metric $g$ in a $C^{\infty}$-generic subset of $\mathcal{F}_{\thin} \cap \Int(\mathcal{T}_{\infty})$ in the sense of Baire, the following holds. For any connected, closed, embedded, $2$-sided minimal hypersurface $S \subset (M, g)$ which is stable, there is a sequence of closed, embedded minimal hypersurfaces $\{\Sigma_{k}\}$ with the following properties: 
    \begin{enumerate}[label=\emph{(\roman*)}]
    \item $\Sigma_{k} \cap S = \emptyset$,
    \item $\lim_{k \rightarrow \infty} ||\Sigma_{k}|| = \infty$,
    \item $\lim_{k \rightarrow \infty}\mathbf{F}\left( \frac{[S]}{||S||}, \frac{[\Sigma_{k}]}{||\Sigma_{k}||} \right) = 0$.
    \end{enumerate}
\end{thm}

\begin{rmk}
    This theorem can be generalized to $1$-sided minimal hypersurface $S$ with stable $2$-sided double cover. Unlike \cite[Theorem 0.1]{Song-Zhou21}, we are unable to provide an explicit quantitative estimate in item (iii) above 
    %on how close the minimal hypersurfaces are from $S$ 
    because this depends on the quantitative behavior of ends. Our proof follows generally the approach in \cite{Song-Zhou21}, but there are essential new challenges that inspire advancement:
    \begin{enumerate}[label=(\alph*)]
        \item Min-max theorem associated with single-cylindrical width and uniform bound of single-cylindrical width (see Section \ref{twopointone} for the definition of single-cylindrical width);
        \item A sequence of compact, almost properly embedded, minimal hypersurfaces produced in compact exhaustions being confined in a compact region (no escape to infinity);
        \item Perturbation of a family of metrics to satisfy ``free boundary bumpyness''. 
    \end{enumerate} 
\end{rmk}

\subsection*{Motivations}  
Minimal hypersurfaces can be viewed as non-linear geometric analogues of the Laplacian eigenfunctions. There is an extensive study of equidistribution and scarring in spectral theory and ergodic theory. While the quantum ergodicity theorem maintains that the normalized densities of a density one subsequence of eigenfunctions individually equidistribute, the Quantum Unique Ergodicity conjecture anticipates that the normalized densities of the full sequence should equidistribute. However, in certain generic situations the $L^{2}$-densities may ``scar'' along subregions rather than equidistribute. In cases where the ambient manifolds are compact, \cite{Song-Zhou21} provides a comprehensive overview of related results \cite{BL67, Zel87, Rudnick-Sarnak94, FSS03, Mar06, Col09, Hassell10, Zel17, KL17}. For complete non-compact ambient manifolds, results are comparatively few, and we give a non-exhaustive list here. In \cite{Zel91}, Zelditch showed that the quantum ergodicity property extends to non-compact hyperbolic surfaces of finite area. Bonthonneau-Zelditch, in their work \cite{BZ16}, gave a new and simpler proof of that result which generalizes to any cusped manifolds with ergodic geodesic flows. In \cite{Bon16}, Bonthonneau proved an Egorov lemma until Ehrenfest times and then gave a version of Quantum Unique Ergodicity for the Eisenstein series under an asymptotic constraint. In \cite{Stu21}, Studnia proved quantum ergodicity for the eigenfunctions of the pseudo-Laplacian on surfaces with hyperbolic cusps and ergodic geodesic flow. Unlike the results above, Marklof showed that under the no extreme level clustering condition, eigenfunctions of the Laplacian on certain non-compact domains in $\mathbb{R}^{2}$ with finite area may localize at infinity and hence rule out Quantum Unique Ergodicity for such systems \cite{Mar05}.  
\\
\\
\noindent \textbf{The class $\mathcal{T}_{\infty}$ of manifolds thick at infinity.} Let us revisit the notion of ``thick at infinity'' put forth by Gromov in \cite{Gro14} and weaken by Song in \cite{Song19}. This condition is crucial for the implementation of White's Manifold Structure Theorem \cite{White91, White17}. We refer to Section \ref{twopointthree} for more details. 

\begin{defn}
    Let $(M, g)$ be a complete $(n + 1)$-dimensional Riemannian manifold. $(M, g)$ is said to be thick at infinity if any connected finite volume complete minimal hypersurface in $(M, g)$ is closed. We denote by $\mathcal{T}_{\infty}$ the class of manifolds that are thick at infinity. 
\end{defn}

\noindent \textbf{The class $\mathcal{F}_{\thin}$ of manifolds with a thin foliation at infinity.} Another key condition is ``thin foliation at infinity'' introduced by Song in \cite{Song19}. This ensures that the $p$-width of any manifold in the class grows sublinearly in $p$ as $p \rightarrow \infty$. We refer to Section \ref{twopointtwo} for more details.   

\begin{defn}
    A complete $(n + 1)$-dimensional Riemannian manifold $(M, g)$ is said to have a thin foliation at infinity if there is a proper Morse function $f: M \rightarrow [0, \infty)$ so that
    \[
    \lim_{t \rightarrow \infty} \Vol_{n}(f^{-1}(t)) = 0. 
    \]
\end{defn}   

Let $M$ be a $(n + 1)$-dimensional smooth manifold. In our setting the natural topology on the space of complete metrics of $M$ is the strong $C^{\infty}$-topology, or ``Whitney $C^{\infty}$-topology'' (see \cite[Chapter 2]{Hir97} or \cite[Chapter II]{GG73} where it is defined for spaces of functions). As in \cite{Song19}, it can be described as follows. Given $b_1, b_2, \ldots$ a sequence of open balls forming a locally finite covering of $M$, any complete metric $g$, a non-negative integer $q$, and $\mathbf{e} = (\epsilon_1, \epsilon_2, \ldots)$ a sequence of positive numbers, set 
\[
O(g, \mathbf{e}, q) := \bigcap_{i = 1}^{\infty} \{g': ||(g' - g)|_{b_i}||_{C^{q}} < \epsilon_i\}. 
\]
Then the strong $C^{q}$-topology (resp. strong $C^{\infty}$-topology) on the space of complete metrics of $M$ has a base of open sets given by $\{O(g, \mathbf{e}, q)\}_{g, \mathbf{e}}$ (resp. $\{O(g, \mathbf{e}, q)\}_{g, \mathbf{e}, q}$). Endowed with the strong topology, the space of complete metrics becomes a Baire space \cite[Chapter 2, Theorem 4.4]{Hir97}. We denote by $\mathcal{F}_{\thin}$ (resp. $\mathcal{T}_{\infty}$) the family of complete metrics on $M$ with a thin foliation at infinity (resp. thick at infinity). Let $\Int(\mathcal{T}_{\infty})$ be the interior of $\mathcal{T}_{\infty}$ in the space of complete metrics. The intersection $\mathcal{F}_{\thin} \cap \Int(\mathcal{T}_{\infty})$ is a non-empty open set in the strong $C^{\infty}$-topology. Note that if $M$ is compact, any metric on $M$ is in $\mathcal{F}_{\thin} \cap \Int(\mathcal{T}_{\infty})$. 

For this class of metrics on $M$, Song proved the following result that extends the density theorem of Irie-Marques-Neves in \cite{Irie-Marques-Neves18}.

\begin{thm}[{\cite[Theorem 1.4]{Song19}}] \label{GODT}   
    Let $M^{n + 1}$ be a smooth manifold of dimension $3 \leq (n + 1) \leq 7$.\\ 
    $(1)$ For any metric $g$ in a $C^{\infty}$-dense subset of $\mathcal{F}_{\thin}$, the union of complete, finite volume, embedded minimal hypersurfaces in $(M, g)$ is dense.\\
    $(2)$ For any metric $g'$ in a $C^{\infty}$-generic subset of $\mathcal{F}_{\thin} \cap \Int(\mathcal{T}_{\infty})$ in the sense of Baire, the union of closed, embedded minimal hypersurfaces in $(M, g')$ is dense.
\end{thm}

We refer to \cite{Song19} for further remarks and examples about different classes of complete non-compact Riemannian manifolds.

Theorems \ref{MAIN} and \ref{GODT} give rise to the following open questions. 

\begin{question}
    Let $M^{n + 1}$ be a smooth manifold of dimension $3 \leq (n + 1) \leq 7$. Is it the case that for any metric $g$ in a $C^{\infty}$-dense subset of $\mathcal{F}_{\thin}$, the following holds: for any connected, closed, embedded, $2$-sided minimal hypersurface $S \subset (M, g)$ that is strictly stable, there is a sequence of complete, finite volume, embedded minimal hypersurfaces $\{\Sigma_{k}\}$ satisfying properties (i), (ii), and (iii) as in Theorem \ref{MAIN}? What conditions on the metric shall we impose to make $\Sigma_{k}$ non-compact?    
\end{question}

\begin{question}
    Does the scarring phenomenon occurs if $S$ is a connected, complete, non-compact, finite volume, embedded $2$-sided minimal hypersurface that is strictly stable? 
\end{question}

\subsection*{Challenges and ideas}
The proof of Theorem \ref{MAIN} relies on a perturbation argument applied to the asymptotic formulas for certain min-max widths. Such a strategy was adopted in \cite{Marques-Neves-Song19} to the Weyl law for the volume spectrum, and in \cite{Song-Zhou21} to the cylindrical Weyl law for widths associated with a given strictly stable minimal surface $S$. In our case, $M$ is in general non-compact and may not obey the Weyl law. To bypass this, we consider a compact exhaustion $\{M_{m}\}$ of $M$, which contains the strictly stable minimal surface $S$ for $m$ large enough. By picking a connected component of $M_{m} \setminus S$ and taking the metric completion, we obtain a compact manifold $\hat{M}_{m}$ with boundary $\partial \hat{M}_{m}$ containing $S$ as a hypersurface. Following the construction of \cite{Song18}, for each large $m$ we form a non-compact manifold $\tilde{\mathcal{C}}(\hat{M}_{m})$ with a single cylindrical end attached along $S$. One technical issue is that the single-cylindrical Weyl law for $\tilde{\omega}_{p}(\hat{M}_{m})$ may depend on $m$, so we might lose control over the remainder terms as $m \rightarrow \infty$. To fix this, we use the assumption $g \in \mathcal{F}_{\thin}$ to show that the growth of $\tilde{\omega}_{p}(\hat{M}_{m})/p$ is sufficiently close to $\Area(S)$ as $m \rightarrow \infty$.  

In \cite{Song-Zhou21}, a key ingredient of the perturbation argument involved the study of a codimension $1$ Banach submanifold in the space of minimal embeddings. In our scenario, the minimal hypersurface produced by $\tilde{\omega}_{p}(\hat{M}_{m})$ may have free boundary in $\partial \hat{M}_{m} \setminus S$. 
%is not closed but with free boundary in $\partial \hat{M}_{m} \setminus S$. 
Therefore, when perturbing the family of metrics on $\hat{M}_{m}$, we need to investigate a codimension $1$ Banach submanifold in the space of minimal embeddings with possibly non-empty free boundary. To achieve this, we analyze the first eigenvalue of an elliptic PDE subject to the boundary condition and apply the Implicit Function Theorem to conclude.  

Another important difference with \cite{Song-Zhou21} is that the perturbation argument is $2$-fold: both global and local. On one hand, we perturb the family of metrics $g_{t}$ on $M$ to $\tilde{g}_{t}$ in the strong $C^{\infty}$-topology. On the other hand, we perturb the family of restrictions $\tilde{g}_t$ on $M_{m}$ to $\tilde{g}_{m, t}$ in the $C^{\infty}$-topology so that $\tilde{g}_{m, t}$ converges smoothly and locally uniformly to $\tilde{g}_{t}$ as $m \rightarrow \infty$. By varying $m$, we obtain a sequence of compact, almost properly embedded free boundary minimal hypersurfaces in $M \setminus S$. A priori this sequence might escape to infinity and no compactness theorem applies to yield a subsequence limit. To resolve this, we observe that the mass of those minimal hypersurfaces are concentrated near $S$. By adjusting the sequence and arguing similarly as in the proof of Theorem \ref{GODT}, we obtain a subsequence limit that is closed, non-degenerate, multiplicity one, and has large area. By the quantitative constancy theorem for complete manifolds, the limit hypersurface approximates $S$ after renormalization. The final scarring result follows from an induction argument on all strictly stable minimal hypersurfaces.      

\subsection*{Organization} The paper is organized as follows. In Section \ref{sectiontwo}, we introduce the single-cylindrical width and the associated min-max theorem. In particular, we derive some asymptotic estimates and a derivative formula of the single-cylindrical width. In Section \ref{sectionthree}, a perturbation argument is employed to confirm a key deformation result. We finish the proof of our generic scarring theorem in Section \ref{sectionfour}. In the ``Appendix'', we give a proof of the quantitative constancy theorem for complete manifolds.      

\subsection*{Acknowledgements}

I am grateful to my advisor Xin Zhou for suggesting this problem. I would like to thank him for his constant support and helpful discussions. Thanks to Antoine Song for reading the draft and useful comments to improve the presentation. Part of this work was done when I visited Princeton University and I would like to thank Princeton University for their hospitality. The work was partially supported by NSF grant DMS-1945178. 

\section{Single-Cylindrical Width and Associated Min-max Theory} \label{sectiontwo}

\subsection{Min-max theory in compact manifolds with boundary} \label{twopointone}

Let $(\hat{M}, \partial \hat{M}, g)$ denote a compact connected $(n + 1)$-dimensional smooth manifold with boundary endowed with a $C^{q}$ metric with $q \geq 3$. Suppose that $\partial \hat{M}$ is a disjoint union of connected hypersurfaces $\{S_{0} = S, S_{1}, \ldots, S_{l}\}$, with the extra assumption that $S$ is a strictly stable minimal hypersurface. Since $S$ admits a strictly mean convex foliation, there is a diffeomorphism 
\[
\Phi: S \times [0, \hat{t}] \rightarrow S
\]
where $\Phi(S \times \{0\}) = S$ is a minimal surface, and for all $t \in (0, \hat{t}]$, the leaf $\Phi(S \times \{t\})$ has non-zero mean curvature vector pointing towards $S$. 

By attaching an infinite cylinder $S \times [0, \infty)$ to $\hat{M}$ via the canonical identity map $\varphi: S \times \{0\} \rightarrow S$, we obtain the following non-compact manifold with a single cylindrical end $\tilde{\mathcal{C}}(\hat{M})$:
\[
\tilde{\mathcal{C}}(\hat{M}) = \hat{M} \cup_{\varphi} (S \times [0, \infty)), 
\]
Note that $\tilde{\mathcal{C}}(\hat{M})$ is a non-compact smooth manifold with boundary endowed with a Lipschitz metric $h$ satisfying $h|_{\hat{M}} = g|_{\hat{M}}$ and $h|_{S \times [0, \infty)} = g|_{S} \oplus dt^{2}$.  

Consider an exhaustion of $\tilde{\mathcal{C}}(\hat{M})$ by compact subsets $K_1 \subset K_2 \subset \cdots \subset \tilde{\mathcal{C}}(\hat{M})$. The number introduced below is well-defined and independent of the choices of the compact exhaustion $\{K_{i}\}$.   

\begin{defn}
    For each positive integer $p$, the single-cylindrical $p$-width of $(\hat{M}, g)$ is defined as 
    \begin{equation}
        \tilde{\omega}_{p}(\hat{M}, g) := \omega_{p}(\tilde{\mathcal{C}}(\hat{M}), h) := \lim_{i \rightarrow \infty} \omega_{p}(K_{i}, h).  
    \end{equation}
    Here, $\omega_{p}(K_{i}, h)$ is the $p$-width of $(K_i, h)$ defined in \cite{Marques-Neves17}. The sequence $\{\tilde{\omega}_{p}(\hat{M}, g)\}_{p}$ will be called single-cylindrical volume spectrum. 
\end{defn}

These widths satisfy a single-cylindrical Weyl law. The proof imitates the proof of the cylindrical Weyl law in \cite{Song18}, which is skipped here.  

\begin{thm}[Single-Cylindrical Weyl Law]\label{WLSC}
    With the notations above, $\tilde{\omega}_{p}(\hat{M}, g)$ is finite for all $p$ and there is a constant $C = C(\hat{M}, g) > 0$ such that 
        \begin{equation}
        p \cdot \Area(S) \leq \tilde{\omega}_{p}(\hat{M}, g) \leq p \cdot \Area(S) + C(\hat{M}, g) p^{\frac{1}{n + 1}}.
        \end{equation}
\end{thm} 

There is a min-max theorem associated with single-cylindrical $p$-width, which in some sense combines the cylindrical min-max theorem \cite{Song18} and the free boundary min-max theorem \cite{Li-Zhou16, Sun-Wang-Zhou20}. To obtain the multiplicity one of free boundary minimal hypersurfaces produced by min-max, it is necessary to impose a condition on the metric $g$, as outlined in \cite{Sun-Wang-Zhou20}:

\begin{defn}
    A metric $g$ on $\hat{M}$ is called strongly bumpy if 
    \begin{enumerate}[label={(\roman*)}]
        \item every finite cover of an immersed free boundary minimal hypersurface is non-degenerate,
        \item every immersed free boundary minimal hypersurface is properly embedded (i.e. the touching set is empty).  
    \end{enumerate}
\end{defn}

By confirming an adapted version of White's Transversality Theorem for free boundary minimal hypersurfaces, Sun-Wang-Zhou proved the following strong bumpy metric theorem, which asserts that the set of strongly bumpy metrics is generic in the Baire sense. As a result, any metric on $\hat{M}$ can be slightly perturbed into a strongly bumpy metric.  

\begin{lem}[{\cite[Theorem 1.3]{Sun-Wang-Zhou20}}]
    For a generic metric on a compact Riemannian manifold with boundary, every embedded free boundary minimal hypersurface is both proper and non-degenerate. 
\end{lem}
 
At this point, we shall present the min-max theorem associated with single-cylindrical width. 

\begin{thm}\label{CMM}
    Let $(\tilde{C}(\hat{M}), h)$ be constructed as above. For each $p \in \mathbb{N}$, there exist a $C^{q}$ compact, almost properly embedded free boundary minimal hypersurface $\Gamma_{p}$ contained in $\hat{M} \setminus S$, whose connected components are called $\Gamma_{p}^{(1)}, \ldots, \Gamma_{p}^{(k_{p})}$, and associated positive integer multiplicities $m_{1}, \ldots, m_{k_{p}}$ such that 
    \begin{align}
    \begin{split}
        &\tilde{\omega}_{p}(\hat{M}, g) = \sum_{i = 1}^{k_p} m_i \Area(\Gamma^{(i)}_{p}), \hspace{10pt} \text{ and }\\
        &\Index(\Gamma_{p}) \leq p. 
    \end{split}
    \end{align}
    By assuming that $g$ is strongly bumpy, we can choose $\Gamma_{p}$ to be properly embedded so that each $\Gamma_{p}^{(i)}$ is $2$-sided and $m_{i} = 1$ for all $i \in \{1, \ldots, k_{p}\}$.    
\end{thm}

\begin{proof}
    The argument for the first part essentially follows the proof of \cite[Theorem 9]{Song18} so we only sketch it here, relying on \cite{Song18} for some details. By varying the metric and resolving singularities around $S$, we form the compact smooth approximations $(U_{\epsilon}, h_{\epsilon})$ of $(\tilde{C}(\hat{M}), h)$. Fix $p \in \mathbb{N}$. Applying the free boundary min-max theorem gives a varifold $V_{\epsilon}$ with $\spt V_{\epsilon} = \Gamma_{\epsilon}$ a $C^{q}$ compact, almost properly embedded free boundary minimal hypersurface such that 
    \[
    \omega_{p}(U_{\epsilon}, h_{\epsilon}) = \mathbf{M}(V_{\epsilon}) =\sum_{i = 1}^{k_{\epsilon}} m_{i, \epsilon} \Area(\Gamma^{(i)}_{\epsilon}). 
    \]
    Here $\Gamma^{(1)}_{\epsilon}, \ldots, \Gamma^{(k_{\epsilon})}_{\epsilon}$ are connected components of $\Gamma_{\epsilon}$ with integer multiplicities $m_{1, \epsilon}, \ldots, m_{k_{\epsilon}, \epsilon}$. By construction, $\partial U_{\epsilon}$ is a disjoint union of connected hypersurfaces $\{S', S_{1}, \ldots, S_{l}\}$. Since $\Phi(\partial{S} \times \{\epsilon\})$ is strictly mean-concave, the monotonicity formula together with the maximum principle implies that $\Gamma_{\epsilon}$ must be compact in $U_{\epsilon} \setminus S'$. 
    
    As $\epsilon \rightarrow 0$, we have $\omega_{p}(U_{\epsilon}, h_{\epsilon}) \rightarrow \omega_{p}(\tilde{C}(\hat{M}), h)$. Then for a sequence $\epsilon_k \rightarrow 0$, the varifold $V_{\epsilon_k}$ converges in the varifold sense to a varifold $V_{p}$ in $\tilde{C}(\hat{M})$ of total mass $\omega_{p}(\tilde{C}(\hat{M}), h)$. By the index bound of Marques-Neves and Sharp's Compactness Theorem, the restriction of $\spt V_p = \Gamma_{p}$ to $\tilde{C}(\hat{M}) \setminus S$ is a $C^{q}$ minimal hypersurface with index at most $p$. The maximum principle implies that if $\Gamma_{p} \cap (\tilde{C}(\hat{M}) \setminus \hat{M}) \neq \emptyset$, $\Gamma_{p}$ would be a connected component of some slice $S \times \{\delta\}$, which contradicts with the strictly mean-concave nature of the foliation. As a consequence, $\Gamma_{p}$ is contained in the compact set $(\hat{M}, g)$. Since $V_{p}$ is a $g$-stationary integral varifold, the maximum principle by White \cite{White10} implies that $V_{p}$ is confined in $\hat{M} \setminus S$. This completes the proof that $\Gamma_{p}$ is a $C^{q}$ compact, almost properly embedded free boundary minimal hypersurface in $\hat{M} \setminus S$ with index at most $p$. 

    When $g$ is strongly bumpy, \cite[Theorem 4.6]{Sun-Wang-Zhou20} implies the existence of a homotopy class $\Pi$ of $p$-sweepouts with $\omega_{p}(U_{\epsilon}, h_{\epsilon}) = \mathbf{L}(\Pi)$. By slightly perturbing the metric $h_{\epsilon}$ to make it strongly bumpy, we obtain from \cite[Theorem 4.7]{Sun-Wang-Zhou20} a minimizing sequence $\{\Phi_{j}\}$ in $\Pi$ such that the critical set $\mathbf{C}(\{\Phi_{j}\})$ contains a $C^{q}$ compact, $2$-sided, multiplicity one, properly embedded free boundary minimal hypersurface $\Gamma_{\epsilon}$. By the compactness theorem \cite[Theorem 5]{Ambrozio-Carlotto-Sharp17}, for a sequence $\epsilon_{k} \rightarrow 0$, $\Gamma_{\epsilon_{k}}$ converges to a $C^{q}$ compact, $2$-sided, multiplicity one, properly embedded free boundary minimal hypersurface $\Gamma_{p}$ because $g$ is strongly bumpy. Since $S$ has no Jacobi field due to bumpiness, all $\Gamma_{\epsilon}$ are contained in $\hat{M}$ and have a uniform positive distance from $S$. By Hausdorff convergence, we conclude that $\Gamma_{p}$ is contained in $\hat{M} \setminus S$.       
\end{proof}

\subsection{Uniform bound of single-cylindrical widths} \label{twopointtwo}

Let $(M^{n + 1}, g)$ be a complete manifold of dimension $3 \leq (n + 1) \leq 7$ with a thin foliation at infinity. There is a proper Morse function $f: M \rightarrow [0, \infty)$ so that 
\[
\lim_{t \rightarrow \infty} \Vol_{n}(f^{-1}(t)) = 0. 
\]
We can assume that positive integers are not critical values of $f$. Then by setting $M_{m} := f^{-1}([0, m])$, we obtain an exhaustion of $M$ by compact domains $M_{1} \subset M_{2} \subset \cdots \subset M$. 

Let $S \subset (M, g)$ be a connected, closed, embedded, $2$-sided minimal hypersurface that is strictly stable. Pick $m_{0} = m_{0}(S)$ large so that $S$ is contained in the interior of $M_{m}$ for all integers $m \geq m_{0}$. We select a connected component of $M_{m} \setminus S$, take the metric completion and obtain a compact Riemannian manifold with boundary $(\hat{M}_{m}, g|_{M_m})$ (the lift of $g|_{M_m}$ to $\hat{M}_{m}$ is still denoted by $g|_{M_m}$). By construction, $\partial \hat{M}_{m}$ is a disjoint union of connected hypersurfaces $\{S_{0} = S, S_{1}, \ldots, S_{l}\}$, where $S$ is a strictly stable minimal hypersurface. The single-cylindrical Weyl law yields that 
\begin{equation}\label{SSWLM}
    p \cdot \Area(S) \leq \tilde{\omega}_{p}(\hat{M}_{m}, g|_{M_m}) \leq p \cdot \Area(S) + C(\hat{M}_{m}, g|_{M_m}) p^{\frac{1}{n + 1}}.
\end{equation}

Note that the constant $C$ in (\ref{SSWLM}) depends on $m$. Our goal is to bound $\tilde{\omega}_{p}(\hat{M}_{m}, g|_{M_m})$ independent of the values of $m$. This becomes essential when we pursue a derivative estimate for the normalized min-max numbers $\tilde{\omega}_{p}(\hat{M}_{m}, g|_{M_m})/p$. To attain the bound, we first refer to a result presented in \cite{Song19}. 

\begin{lem}\label{SAB}
    Let $(M^{n + 1}, g)$ be a complete manifold of dimension $3 \leq (n + 1) \leq 7$ with a thin foliation at infinity. Then $\omega_{p}(M, g)$ is finite for all $p$ and admits a sublinear asymptotic behavior:
    \begin{equation}
        \lim_{p \rightarrow \infty} \frac{\omega_{p}(M, g)}{p} = 0. 
    \end{equation}
    If we write $\omega_{p}(M, g) = f(p)$, then $f(p)$ is finite for all $p$ and $\lim_{p \rightarrow \infty} f(p)/p = 0$. Note that $f(p)$ is locally bounded in the set of $C^{q}$ metrics on $M$ endowed with the strong $C^1$-topology and satisfies $\lim_{p \rightarrow \infty} f(p)/p = 0$. 
\end{lem}

\begin{proof}
    Since $(M, g) \in \mathcal{F}_{\thin}$, for any $\mu > 0$, there is a compact subset $K_{\mu} \subset M$ so that any bounded domain of $M \setminus K_{\mu}$ has a thin foliation $\{\Sigma_{t}\}_{t \in [0, 1]}$ such that 
    \[
    \sup_{t \in [0, 1]} \Vol_{n}(\Sigma_{t}) \leq \mu. 
    \]
    We need to show that $\omega_{p}(M, g)$ is finite and asymptotically sublinear. 
    
    Let $D_{\mu}$ be a compact region of $M$ containing $K_{\mu}$. Then the region $D_{\mu} \setminus K_{\mu}$ has a $1$-sweepout $\{\Sigma_{t}\}_{t \in [0, 1]}$ such that 
    \[
    \sup_{t \in [0, 1]} \Vol_{n}(\Sigma_{t}) \leq \mu. 
    \]  
    By the proof of Claim 5.6 of \cite[Theorem 5.1]{Marques-Neves17}, for each $p \geq 1$ we can construct from $\{\Sigma_{t}\}$ a $p$-sweepout $\Psi_{p}: \mathbb{RP}^{p} \rightarrow \mathcal{Z}_{n, \rel}(D_{\mu} \setminus K_{\mu}, \partial (D_{\mu} \setminus K_{\mu}); \mathbb{Z}_{2})$ which has no concentration of mass such that
    \[
    \sup_{x \in \mathbb{RP}^{p}} \mathbf{M}(\Psi_{p}(x)) \leq p \mu. 
    \]
    By the sublinear bound on the $p$-widths \cite{Marques-Neves17} which holds for compact manifolds with boundary, there is a $p$-sweepout $\Phi_{p}: \mathbb{RP}^{p} \rightarrow \mathcal{Z}_{n, \rel}(K_{\mu}, \partial K_{\mu}; \mathbb{Z}_{2})$ which has no concentration of mass and there is a constant $C = C(K_{\mu}, g)$ independent of $p$ such that 
    \[
    \sup_{x \in \mathbb{RP}^{p}} \mathbf{M}(\Phi_{p}(x)) \leq Cp^{\frac{1}{n + 1}}. 
    \]

    Using the gluing technique of Liokumovich-Marques-Neves \cite{Liokumovich-Marques-Neves16}, we add tiny tubes to connect the regions and glue the $p$-sweepouts $\Phi_{p}$ and $\Psi_{p}$ together to obtain a $p$-sweepout $\hat{\Phi}_{p}: \mathbb{RP}^{p} \rightarrow \mathcal{Z}_{n}(D_{\mu}; \partial D_{\mu}; \mathbb{Z}_{2})$ satisfying 
    \begin{align*}
        \max_{x \in \mathbb{RP}^{p}} \mathbf{M}(\hat{\Phi}_{p}(x)) &\leq \max_{x \in \mathbb{RP}^{p}} \mathbf{M}(\Phi_{p}(x)) +  
        \max_{x \in \mathbb{RP}^{p}} \mathbf{M}(\Psi_{p}(x)) + \Area(\partial K_{\mu})\\ 
        &\leq C p^{\frac{1}{n + 1}} +  p\mu + \mu\\
        &\leq \hat{C} p^{\frac{1}{n + 1}} + p\mu,
    \end{align*}
    where $\hat{C} = C + \mu$. Taking $D_{\mu}$ arbitrarily large, we get 
    \[
    \omega_{p}(M, g) \leq \hat{C} p^{\frac{1}{n + 1}} +  p\mu,
    \]
    which shows that $\omega_{p}(M, g)$ is finite for all $p$. Since $\mu$ is arbitrarily small, we conclude that 
    \[
    \lim_{p \rightarrow \infty} \frac{\omega_{p}(M, g)}{p} = 0,
    \]
    i.e. $\omega_{p}(M, g)$ is asymptotically sublinear. 
\end{proof}

Since the proof of Theorem \ref{DT} requires a local perturbation of a family of metrics on $M_m$, we need the following setup in addition to Lemma \ref{SAB}. Let $(M, g), S$, and $M_m$ be as above. Assume that each $M_{m}$ admits a $C^{q}$ metric $g_{m}$ with $q \geq 3$ so that $g_{m}$ converges smoothly and locally uniformly to $g$. Then for $m \geq m_{0}$ with the constant $m_{0}$ enlarged, the metric $g_{m}$ lies in a small $C^{q}$-neighborhood of $g|_{M_m}$. By isometrically embedding $M_m$ into a closed manifold $\tilde{M}_{m}$ and applying the Implicit Function Theorem on $\tilde{M}_{m}$, there exists a unique minimal hypersurface $S_{g_{m}}$ in the interior of $M_m$ as a section of the normal bundle of $S$. Moreover, $S_{g_{m}}$ is $2$-sided and strictly stable. By selecting a connected component of $M_{m} \setminus S_{g_m}$ and taking the metric completion, we obtain a compact Riemannian manifold with boundary $(\hat{M}_{m, g_{m}}, g_{m})$ (the lift of $g_m$ to $\hat{M}_{m, g_m}$ is still denoted by $g_m$). Our goal is to bound $\tilde{\omega}_{p}(\hat{M}_{m, g_{m}}, g_{m})$ independent of the values of $m$. 

\begin{thm}\label{ALUB}
    Consider the setup above. The single-cylindrical $p$-width admits the following bound for $m \geq \max\{m_{0}, m_{1}\}$:
        \begin{equation}\label{fubsc}
        p \cdot \Area_{g_m}(S_{g_m}) \leq \tilde{\omega}_{p}(\hat{M}_{m, g_{m}}, g_{m}) \leq p \cdot \Area_{g_{m}}(S_{g_m}) +  \Area_{g_m}(S_{g_m}) + f(p) + 2.
        \end{equation}
    Moreover, for any $\epsilon > 0$, we have the uniform bound for $m \geq \max\{m_{0}, m_{1}, m_{2}\}$: 
        \begin{equation}\label{subsc}
        \Area_{g}(S) - \epsilon \leq \frac{1}{p}\tilde{\omega}_{p}(\hat{M}_{m, g_{m}}, g_{m}) \leq  \Area_{g}(S) +  \epsilon + \frac{\Area_{g}(S) + \epsilon + f(p) + 2}{p}.
        \end{equation}
    In both statements, $f(p)$ is locally bounded in the set of $C^{q}$ metrics on $M$ endowed with the strong $C^1$-topology and satisfies $\lim_{p \rightarrow \infty} f(p)/p = 0$. 
\end{thm}

\begin{proof}
    Consider (\ref{fubsc}) first. As proven in \cite[Theorem 9]{Song18}, we have $\omega_{1}(S_{g_m} \times [0, L], h) = \Area_{g_{m}}(S_{g_m})$ when $L$ is large enough, where $h$ satisfies $h|_{\hat{M}_{m, g_{m}}} = g_m$ and $h|_{S_{g_m} \times [0, \infty)} = g_{m}|_{S_{g_m}} \oplus dt^2$. The lower bound is obtained by using a Lusternik-Schnirelmann type argument. For all $R > 0$ large enough and $x \in \tilde{\mathcal{C}}(\hat{M}_{m, g_{m}})$, we have $S_{g_m} \times [0, L] \subset B_{R}(x)$ such that   
    \[
    \tilde{\omega}_{1}(\hat{M}_{m, g_{m}}, g_{m}) \geq \omega_{1}(B_{R}(x), h) \geq \omega_{1}(S_{g_m} \times [0, L], h) = \Area_{g_{m}}(S_{g_m}).  
    \] 
    Now, fix $x_{0} \in \tilde{\mathcal{C}}(\hat{M}_{m, g_m})$. Given $\delta > 0$, choose $R_{p}$ large enough such that 
    \[
    \omega_{p}(B_{R_{p}}(x_{0}), h) \geq \tilde{\omega}_{p}(\hat{M}_{m, g_m}, g_{m}) - \delta. 
    \]
    Due to the fact that $B_{R_{p}}(x_{0}) \sqcup (S_{g_m} \times [0, L]) \subset \tilde{\mathcal{C}}(\hat{M}_{m, g_m})$, the Lusternik-Schnirelmann inequality gives 
    \[
    \tilde{\omega}_{p + 1}(\hat{M}_{m, g_m}, g_{m}) \geq \omega_{p}(B_{R_{p}}(x_{0}), h) + \omega_{1} (S_{g_m} \times [0, L], h) \geq \tilde{\omega}_{p}(\hat{M}_{m, g_m}, g_{m}) + \Area_{g_m}(S_{g_m}) - \delta.  
    \]
    Since $\delta$ is arbitrary, we prove the lower bound.  

    To attain the upper bound, we first note that $\hat{M}_{m, g_m} \subset M$ is a compact domain of $M$, and hence by Lemma \ref{SAB} and the fact that $g_{m} \rightarrow g$ smoothly and locally uniformly, the following inequality holds for each $p \in \mathbb{N}$ and all $m \geq m_{1}$:  
    \[
    \omega_{p}(\hat{M}_{m, g_m}, g_{m}) \leq \omega_{p}(M, g) + 1 = f(p) + 1.  
    \]
    
    Let $\theta: S_{g_m} \times [0, L] \rightarrow \mathbb{R}$ be the Morse function defined by $\theta(x, t) := t$. Consider $\Psi_{p}: \mathbb{RP}^{p} \rightarrow \mathcal{Z}_{n}(S_{g_m} \times [0, L]; \partial (S_{g_m} \times [0, L]); \mathbb{Z}_{2})$ defined by
    \[
    \Psi_{p}([a_{0}, a_{1}, \ldots, a_{p}]) = \partial\{x: a_{0} + a_{1} \theta(x) + \cdots + a_{p} \theta(x)^{p} < 0\}. 
    \]
    Then $\Psi_{p}$ is a $p$-sweepout which has no concentration of mass and satisfies
    \[
    \max_{x \in \mathbb{RP}^{p}} \mathbf{M}(\Psi_{p}(x)) \leq p \cdot \Area_{g_m}(S_{g_m}). 
    \] 
    Since $\omega_{p}(\hat{M}_{m, g_m}, g_{m}) \leq f(p) + 1$, there is a $p$-sweepout $\Phi_{p}: \mathbb{RP}^{p} \rightarrow \mathcal{Z}_{n, \rel}(\hat{M}_{m, g_m}, \partial \hat{M}_{m, g_m}; \mathbb{Z}_{2})$ which has no concentration of mass such that
    \[
    \sup_{x \in \mathbb{RP}^{p}} \mathbf{M}(\Phi_{p}(x)) \leq f(p) + 2.
    \]
    
    Using the gluing technique of Liokumovich-Marques-Neves \cite{Liokumovich-Marques-Neves16}, we add tiny tubes to connect the regions and glue the $p$-sweepouts $\Phi_{p}$ and $\Psi_{p}$ together to obtain a $p$-sweepout $\hat{\Phi}_{p}: \mathbb{RP}^{p} \rightarrow \mathcal{Z}_{n}(\hat{M}_{m, g_m} \cup (S_{g_m} \times [0, L]); \partial (\hat{M}_{m, g_m} \cup (S_{g_m} \times [0, L])); \mathbb{Z}_{2})$ satisfying 
    \begin{align*}
        \max_{x \in \mathbb{RP}^{p}} \mathbf{M}(\hat{\Phi}_{p}(x)) &\leq \max_{x \in \mathbb{RP}^{p}} \mathbf{M}(\Phi_{p}(x)) +  
        \max_{x \in \mathbb{RP}^{p}} \mathbf{M}(\Psi_{p}(x)) + \Area_{g_m}(S_{g_m})\\ 
        &\leq f(p) + 2 +  p \cdot \Area_{g_m}(S_{g_m}) + \Area_{g_m}(S_{g_m}).
    \end{align*}
    Letting $L \rightarrow \infty$, we get 
    \[
    \tilde{\omega}_{p}(\hat{M}_{m, g_m}, g_{m}) \leq p \cdot \Area_{g_m}(S_{g_m}) + \Area_{g_m}(S_{g_m}) + f(p) + 2,
    \]
    which completes the proof of (\ref{fubsc}). 
    
    To obtain (\ref{subsc}), we note that for any $\epsilon > 0$, there is a large number $m_{2}$ such that for all integers $m \geq m_{2}$, 
    \[
    |\Area_{g_{m}}(S_{g_m}) - \Area_{g}(S)| \leq \epsilon. 
    \]
    Substituting this into (\ref{fubsc}) and dividing both sides by $p$ yield (\ref{subsc}).  
\end{proof}

\subsection{Single-cylindrical widths are locally Lipschitz functions of the metric} \label{twopointthree} 

The following lemma indicates that for fixed $m \geq m_0$ the normalized single-cylindrical widths $\tilde{\omega}_{p}(\hat{M}_{m, g_{m}}, g_{m})$ are Lipschitz functions of the metric $g_{m}$ within a small $C^{q}$-neighborhood of $g|_{M_m}$. 

\begin{lem}\label{scwll}
    Fix a positive integer $q \geq 3$ and let $g \in \mathcal{F}_{\thin} \cap \Int(\mathcal{T}_{\infty})$ be a $C^{q}$ metric on $M$. Assume that $S \subset (M, g)$ is a connected, closed, embedded, 2-sided minimal hypersurface that is strictly stable. Then there exists $\epsilon > 0$ and $C = C(M_{m}, g, S, q) > 0$ such that for any $C^{q}$ metrics $g_{m, 1}$ and $g_{m, 2}$ on $M_{m}$ with $||g_{m, i} - g|_{M_m}||_{C^{q}} < \epsilon$ for $i = 1, 2$, we have 
    \[
    \left| \frac{1}{p} \tilde{\omega}_{p}(\hat{M}_{m, g_{m, 1}}, g_{m, 1}) - \frac{1}{p} \tilde{\omega}_{p}(\hat{M}_{m, g_{m, 2}}, g_{m, 2}) 
    \right| \leq C ||g_{m, 1} - g_{m, 2}||_{C^{q}}, 
    \]
    where the lifts of $g_{m, 1}, g_{m, 2}$ to $\hat{M}_{m, g_{m, 1}}, \hat{M}_{m, g_{m, 2}}$ are still denoted by $g_{m, 1}, g_{m, 2}$ respectively.   
\end{lem}

\begin{proof}
    It is essentially a computation appearing in the proof of \cite[Lemma 1.4]{Song-Zhou21} and \cite[Lemma 1]{Marques-Neves-Song19}, which is skipped here.
\end{proof}

We end this section with a formula for the derivative of the single-cylindrical width whose proof is similar to that of \cite[Lemma 2]{Marques-Neves-Song19}.  

\begin{lem}\label{DSCW}
    Fix $m \geq m_0$ and a positive integer $q \geq 3$. Let $g \in \mathcal{F}_{\thin} \cap \Int(\mathcal{T}_{\infty})$ be a $C^{q}$ metric on $M$. Assume that $S \subset (M, g)$ is a connected, closed, embedded, 2-sided minimal hypersurface that is strictly stable. Let $\{g_{t}\}_{t \in [0, 1]}$ be a smooth family of $C^{q}$ metrics with $g_{0} = g$. Let $\{g_{m, t}\}_{t \in [0, 1]}$ be a smooth family of $C^{q}$ metrics on $M_{m}$ with $g_{m, 0} = g_{m}$ such that $g_{m, t}$ converges smoothly and locally uniformly to $g_{t}$ for each fixed $t$. For $t$ small, consider $(\hat{M}_{m, g_{m, t}}, g_{m, t})$ and $S_{g_{m, t}}$ constructed in Section \ref{twopointtwo} whenever $m \geq m_{0}$. Assume that 
    \begin{enumerate}[label=\emph{(\roman*)}]
        \item every almost properly embedded free boundary minimal hypersurface in $(\hat{M}_{m, g_{m}}, g_{m})$ is non-degenerate,  
        \item for any connected, $1$-sided, almost properly embedded free boundary minimal hypersurface in $(\hat{M}_{m, g_{m}}, g_{m})$, its $2$-sided double cover has no positive Jacobi field.
    \end{enumerate}
    Suppose also that, for some integer $p$, the single-cylindrical $p$-width function 
    \[
    t \rightarrow \tilde{\omega}_{p}(\hat{M}_{m, g_{m, t}}, g_{m, t})
    \]
    is differentiable at time $t = 0$ for all $m \geq m_{0}$. 

    Then for each $m \geq m_{0}$ there exists a $C^{q}$ compact, almost properly embedded free boundary minimal hypersurface $\Gamma_{m, p}$ contained in $\hat{M}_{m, g_{m}} \setminus S_{g_{m}}$ such that 
    \begin{align*}
        &\tilde{\omega}_{p}(\hat{M}_{m, g_{m}}, g_{m}) = \Area_{g_{m}}(\Gamma_{m, p}), \hspace{10pt} \Index(\Gamma_{m, p}) \leq p, \\
        &\text{ and } \frac{d}{dt}\bigg|_{t = 0} \tilde{\omega}_{p}(\hat{M}_{m, g_{m, t}}, g_{m, t}) = \int_{\Gamma_{m, p}} \frac{1}{2} \Tr_{\Gamma_{m, p}, g_{m}} (\frac{\partial g_{m, t}}{\partial t}\bigg|_{t = 0}) d\Gamma_{m, p}. 
    \end{align*}
\end{lem}

\begin{proof}
    Fix $m \geq m_{0}$. Take a sequence $\{t_{j}\}_{j \in \mathbb{N}} \subset [0, 1]$ with $t_{j} \rightarrow 0$. For each $j$, we find a smooth metric $h_{m, j}$ on $M_{m}$ arbitrarily close to $g_{m, t_j}$ in the $C^{q}$-topology so that $h_{m, j}$ (the lift of $h_{m, j}$ to $\hat{M}_{m, h_{m, j}}$ is still denoted by $h_{m, j}$) is strongly bumpy with 
    \begin{align*}
        &\frac{d}{dt}\bigg|_{t = 0} \tilde{\omega}_{p}(\hat{M}_{m, g_{m, t}}, g_{m, t}) 
        = \lim_{j \rightarrow \infty} \frac{\tilde{\omega}_{p}(\hat{M}_{m, h_{m, j}}, h_{m, j}) - \tilde{\omega}_{p}(\hat{M}_{m, g_{m}}, g_m)}{t_{j}}\\
        &\text{ and } \frac{\partial g_{m, t}}{\partial t}\bigg|_{t = 0} = \lim_{j \rightarrow \infty} \frac{h_{m, j} - g_{m}}{t_{j}}. 
    \end{align*}
    By Theorem \ref{CMM}, for each $j$ there is a $2$-sided, compact, properly embedded free boundary minimal hypersurface $\Gamma'(m, j) \subset \hat{M}_{m, h_{m, j}} \setminus S_{h_{m, j}}$ such that 
    \[
    \tilde{\omega}_{p}(\hat{M}_{m, h_{m, j}}, h_{m, j}) = \Area_{h_{m, j}}(\Gamma'(m, j)), \hspace{10pt} \Index(\Gamma'(m, j)) \leq p. 
    \]
    By compactness theorem \cite[Theorem 5]{Ambrozio-Carlotto-Sharp17}, after taking a subsequence as $j \rightarrow \infty$, $\Gamma'(m, j)$ converges in the varifold sense to a compact, almost properly embedded free boundary minimal hypersurface $\Gamma_{m, p}$ in $(\hat{M}_{m, g_{m}}, g_m)$. If the convergence was not multiplicity one, then either there is a $2$-sided component of $\Gamma_{m, p}$ with a nontrivial Jacobi field or a $1$-sided component of $\Gamma_{m, p}$ whose $2$-sided double cover has a positive Jacobi field. By our assumptions on $g_{m}$, both cases lead to a contradiction. Hence, $\Gamma_{m, p}$ is multiplicity one. Since $\Gamma_{m, p}$ may not be properly embedded, the convergence is locally smooth and graphical outside a finite set. Since $S_{g_m}$ is strictly stable, by Hausdorff convergence, we have $\Gamma_{m, p} \subset \hat{M}_{m, g_{m}} \setminus S_{g_{m}}$ with
    \[
    \tilde{\omega}_{p}(\hat{M}_{m, g_{m}}, g_m) = \Area_{g_m}(\Gamma_{m, p}), \hspace{10pt} \Index(\Gamma_{m, p}) \leq p.
    \]
    By assumption on $g_{m}$, the formula for the derivative follows as in \cite[Lemma 2]{Marques-Neves-Song19}.   
\end{proof}

\subsection{A quantitative constancy theorem} \label{twopointfour} 

In this section, we generalize the quantitative constancy theorem \cite[Theorem 2.1]{Song-Zhou21} to complete manifolds. Recall that in closed case, the quantitative constancy theorem suggests that if the total mass of a stationary $n$-varifold is mostly concentrated in a tubular neighborhood of a closed embedded hypersurface, then the varifold distance between the normalized varifold and the hypersurface has a quantitative estimate. 

Let $(M, g)$ be a complete Riemannian manifold, $g$ a $C^{q}$ metric ($q \geq 3$), and $S \subset M$ a connected, closed, embedded, $2$-sided hypersurface. For $\epsilon > 0$ small enough, let $N_{\epsilon}(S)$ be an $\epsilon$-tubular neighborhood of $S$ so that the normal exponential map of $S$ has no focal point in $N_{\epsilon}(S)$. After choosing a unit normal vector field $\nu $ along $S$, we consider a foliation $\{S_{t}\}_{-\epsilon < t < \epsilon}$ of $N_{\epsilon}(S)$ by a family of equidistant hypersurfaces: 
\[
S_{t} := \{\exp_{p}(t\nu(p)): p \in S\}, 
\]
where $t: N_{\epsilon}(S) \rightarrow (-\epsilon, \epsilon)$ is the signed distance to $S$. Denote the nearest point projection map by 
\[
\pi: N_{\epsilon}(S) \rightarrow S. 
\]
For $x \in N_{\epsilon}(S)$, let $S_{x}$ be the leaf in $\{S_{t}\}$ containing $x$.  

Let $G_{n}(U)$ be the $G(n + 1, n)$-Grassmanian bundle of unoriented $n$-planes over a subset $U \subset M$. Given an $n$-varifold $V$ on $U$ (i.e. a Radon measure on $G_{n}(U)$), we denote by $||V||$ its mass, $||V||(U)$ the mass of the restriction of $V$ to $U$, and $\mu_{V}$ the Radon measure on $M$ associated with $V$. In \cite{Pitts81}, the $\mathbf{F}$-distance between two $n$-varifolds $V, W$ in $M \hookrightarrow \mathbb{R}^{P}$ is defined extrinsically as
\[
\mathbf{F}(V, W) := \sup\{|V(f) - W(f)|; f: G_{n}(\mathbb{R}^{P}) \rightarrow \mathbb{R}, |f| \leq 1, \Lip(f) \leq 1\}. 
\]
Using the distance function $\dist_{g}$ on $G(n + 1, n)$ induced from the unit tangent bundle $UTM$ of $M$, we can intrinsically describe the varifold distance, which aligns with Pitts' definition up to a constant. 

\begin{thm}\label{qctc}
    Let $(M, g)$, $S$, $\{S_{t}\}$, $N_{\epsilon}(S)$ be as above. Then there exists a constant $C = C(g, S)$ such that for any $0 < \delta < \epsilon$, and for any stationary $n$-varifold $V$ in $(M, g)$ satisfying 
    \begin{equation}\label{qct}
    \frac{||V||(M \setminus N_{\delta}(S))}{||V||} \leq \delta, 
    \end{equation}
    we have 
    \[
    \mathbf{F}\left(\frac{[S]}{||S||}, \frac{[V]}{||V||}\right) \leq C \sqrt{\delta}. 
    \]
    The constant $C = C(g, S)$ can be chosen to be uniformly bounded in the set of pairs $(g, S)$ endowed with the product $C^{3}$-topology.  
\end{thm}

The proof imitates the proof of \cite[Theorem 2.1]{Song-Zhou21}. For the sake of completeness, we include but postpone it to ``Appendix''.

\section{Metric deformation and approximation by minimal hypersurfaces} \label{sectionthree}
Let $M^{n + 1}$ be a connected manifold of dimension $3 \leq (n + 1) \leq 7$ endowed with a smooth metric $g \in \mathcal{F}_{\thin} \cap \Int(\mathcal{T}_{\infty})$. The main result of this section is a deformation theorem partially generalizing \cite[Theorem 3.2]{Song-Zhou21}, which asserts that given any connected, closed, embedded, $2$-sided, strictly stable minimal hypersurface $S$, by perturbing the metric slightly one can construct a minimal hypersurface of arbitrarily large area, approximating $S$ after renormalization.   

To begin with, we recall some facts related to White's Manifold Structure Theorem for the space of minimal submanifolds with (possibly empty) free boundary. Let $\tilde{M}$ be a connected closed manifold and let $N$ be a $2$-sided, closed, embedded hypersurface in $\tilde{M}$. Let $\tilde{\Gamma}^{(q)}$ ($q \geq 7$ from now on) be the set of $C^{q}$ metrics on $\tilde{M}$ and denote 

\begin{align*}
\widetilde{\mathcal{M}} := \{(\gamma, \Sigma): \hspace{3pt}&\gamma \in \tilde{\Gamma}^{(q)} \text{ and } (\Sigma, \partial \Sigma) \text{ is an embedded minimal hypersurface in $(\tilde{M}, \gamma)$}\\
&\text{with (possibly empty) free boundary } \partial \Sigma \subset N\}. 
\end{align*}

By the work of \cite{White87, White91} and \cite{Ambrozio-Carlotto-Sharp17}, this space has the structure of a separable $C^{2}$ Banach manifold. The natural projection $\tilde{\Pi}: \widetilde{\mathcal{M}} \rightarrow \tilde{\Gamma}^{(q)}$ defined by $\tilde{\Pi}(\gamma, \Sigma) = \gamma$ is a $C^{2}$ Fredholm map with Fredholm index $0$. All regular points of $\tilde{\Pi}$ form the space
\[
\widetilde{\mathcal{M}}_{\reg} := \{(\gamma, \Sigma) \in \widetilde{\mathcal{M}}: \Sigma \text{ is non-degenerate}\}, 
\]
which is a countable union of open sets $\mathcal{U}_{i}$ such that $\tilde{\Pi}$ maps each $\mathcal{U}_{i}$ homeomorphically onto an open subset of $\tilde{\Pi}$. Given a regular pair $(\gamma, \Sigma) \in \widetilde{\mathcal{M}}_{\reg}$, if $\Sigma$ is $1$-sided, its connected $2$-sided double cover may still carry positive Jacobi field. Set 

\begin{align*}
\widetilde{\mathcal{M}}_{\sided} := \{(\gamma, \Sigma) \in \widetilde{\mathcal{M}}_{\reg}: \hspace{3pt}&\Sigma \text{ is $1$-sided, but its $2$-sided double cover} \\
&\text{carries a positive Jacobi field}\}  
\end{align*}

\noindent and we prove a free boundary version of \cite[Lemma 3.1]{Song-Zhou21}. 

\begin{lem}
    The set $\widetilde{\mathcal{M}}_{\sided}$ is a $C^{2}$ Banach submanifold of $\widetilde{\mathcal{M}}$, of codimension $1$. 
\end{lem}

\begin{proof}
    We cover $\widetilde{\mathcal{M}}_{\reg}$ by countably many open sets $\{\mathcal{U}_{i}\}$ such that $\tilde{\Pi}$ maps each $\mathcal{U}_{i}$ diffeomorphically onto its image $\tilde{\Pi}(\mathcal{U}_{i}) \subset \tilde{\Gamma}^{(q)}$. Given $(\gamma, \Sigma) \in \widetilde{\mathcal{M}}_{\sided} \cap \mathcal{U}$, in a neighborhood $\mathcal{V} \subset \mathcal{U}$ of $(\gamma, \Sigma)$, all other pairs $(\gamma', \Sigma')$ come from $C^2$ graphical deformations of $(\gamma, \Sigma)$ and hence are $1$-sided. If $\Sigma$ has empty boundary, then all other pairs in $\mathcal{V}_{1}$ have empty boundary and the argument in \cite{Song-Zhou21} implies that $ \widetilde{\mathcal{M}}_{\sided} \cap \mathcal{V}_{1}$ is a codimension $1$, $C^{2}$ Banach submanifold of $\mathcal{V}_{1}$. Otherwise, $\Sigma$ has free boundary in $N$ and all other pairs in $\mathcal{V}_{2}$ have free boundary in $N$. Our goal is to show that $\widetilde{\mathcal{M}}_{\sided} \cap \mathcal{V}_{2}$ is a codimension $1$, $C^{2}$ Banach submanifold of $\mathcal{V}_{2}$. Consider the following functional defined in a neighborhood of $\tilde{\Pi}(\mathcal{V}_{2})$: 
    \[
    \tilde{\lambda}_1: \gamma' \mapsto \tilde{\lambda}_{1}(\gamma', \tilde{\Sigma}'),  
    \]
    where $\tilde{\lambda}_{1}(\gamma', \tilde{\Sigma}')$ is the first eigenvalue of the Jacobi operator of the $2$-sided double cover $\tilde{\Sigma}'$ of $\Sigma'$. Note that the coefficients of the Jacobi operator and the boundary condition depend in a $C^{2}$ manner with respect to $\gamma' \in \tilde{\Gamma}^{(q)}$. Since the Jacobi operator is self-adjoint with respect to a volume measure depending on $\gamma'$ and the boundary condition, $\tilde{\lambda}_1: \tilde{\Pi}(\mathcal{V}_{2}) \subset \tilde{\Gamma}^{(q)} \rightarrow \mathbb{R}$ is a $C^2$ map (see \cite{Uhlenbeck76} and ``Appendix'' of \cite{Song-Zhou21} for more details).
    \\   
    \\
    \noindent \textbf{Claim.} The differential of $\tilde{\lambda}_{1}$ is nonzero at metric in $\tilde{\Pi}(\widetilde{\mathcal{M}}_{\sided} \cap \mathcal{V}_{2}) \subset \tilde{\Gamma}^{(q)}$.  \\
    \\
    \noindent \textit{Proof of Claim.} Given $(g, \Sigma) \in \widetilde{\mathcal{M}}_{\sided} \cap \mathcal{V}_{2}$ and $\tilde{\Sigma}$, choose a small open set $U \subset \tilde{M}$ away from $\partial \Sigma$ such that $\Sigma \cap U$ is a $2$-sided $n$-ball with a local unit normal $\nu$. Consider the $1$-parameter smooth perturbation $g_{t} = e^{2tf} g$, where $f \in C^{q - 1}_{c}(U)$ satisfies 
    \begin{enumerate}[label=(\roman*)]
        \item $f = 0, Df = 0, \Hess f(\nu, \nu) \geq 0$ on $U \cap \Sigma$,
        \item $\Hess f(\nu, \nu) > 0$ on some proper open subset of $U \cap \Sigma$.
    \end{enumerate}
    Under $g_{t}$, $\Sigma$ remains as a non-degenerate minimal hypersurface with free boundary in $N$ and the restrictions $(g_{t})|_{\Sigma}$ remain unchanged. We denote by $\tilde{\Sigma}$ the connected $2$-sided double cover of $\Sigma$ and $\varphi_{1, t}$ the first normalized eigenfunction associated to $\tilde{\lambda}_{1}(g_t, \tilde{\Sigma})$ under $g_{t}$. Note that $\varphi_{1, t}$ satisfies the following elliptic PDE subject to the boundary condition:   
    \begin{equation}\label{fbpde}
        \begin{cases}
            L_{t}\varphi_{1, t}  = - \tilde{\lambda}_{1}(g_t, \tilde{\Sigma}) \varphi_{1, t} \hspace{15pt} \text{ on } \tilde{\Sigma}\\
            \frac{\partial \varphi_{1, t}}{\partial \eta} = h(\vec{n}, \vec{n}) \varphi_{1, t} \hspace{39pt} \text{ on } \partial \tilde{\Sigma}
        \end{cases},
    \end{equation}
    where $L_{t} = \Delta_{t} + \Ric_{t}(\nu, \nu) + |A_{t}|^2$ is the Jacobi operator of $\tilde{\Sigma}$ under $g_{t}$, $\eta$ is the outward unit co-normal to $\partial \tilde{\Sigma}$ invariant under $g_{t}$, and $h$ is the second fundamental form of $N$ with respect to the unit normal $\vec{n}$ invariant under $g_t$. As $\varphi_{1, t}$ can be chosen to be $C^{1}$ in $t$, a simple calculation leads to 
    \begin{align}\label{fefb}
    \begin{split}
        \tilde{\lambda}_{1}(g_{t}, \tilde{\Sigma}) =\hspace{3pt}& \int_{\tilde{\Sigma}} |\nabla \varphi_{1, t}|^2 - (\Ric_{t}(\nu, \nu) + |A_{t}|^2) \varphi_{1, t}^2 - \int_{\partial \tilde{
        \Sigma}} h(\vec{n}, \vec{n}) \varphi_{1, t}^2
        \\
        =\hspace{3pt}& \int_{\tilde{\Sigma}} |\nabla \varphi_{1, t}|^2 - (\Ric(\nu, \nu) + |A|^2) \varphi_{1, t}^2 - \int_{\partial \tilde{\Sigma}} h(\vec{n}, \vec{n}) \varphi_{1, t}^2\\
        &+ t \int_{\tilde{\Sigma}} (n - 1) \Hess f(\nu, \nu) \varphi_{1, t}^{2}.  
    \end{split}    
    \end{align}
    Since $\varphi_{1, 0}$ is the eigenfunction associated to $\tilde{\lambda}_{1}(g, \tilde{\Sigma}) = 0$, the $t$-derivatives of the first two terms add up to $0$. The $t$-derivative of the remaining term is positive by assumption on $f$ and because $\varphi_{1, 0}$ is nowhere vanishing. Approximating $g_{t}$ by a smooth variation $\tilde{g}_t$ of $C^{q}$ metrics completes the proof of the claim.\\
    \\
    \indent By the Implicit Function Theorem, we deduce that the intersection $\widetilde{\mathcal{M}}_{\sided} \cap \mathcal{V}_{2}$ has codimension $1$ in $\mathcal{V}_{2}$. Since there are only countably many $1$-sided pairs for each fixed $\gamma$, we finish the proof.    
\end{proof}

Let $g_{t}$ be a smooth $1$-parameter family of metrics on $M$ so that for each fixed $m \geq m_{0}$ the restriction $g_{t}|_{M_m}: [0, 1] \rightarrow \Gamma^{(q)}$ is a smooth $1$-parameter family of metrics on $M_m$. We may isometrically embed $(M_m, \partial M_m, g_{t}|_{M_m})$ into a connected closed manifold $(\tilde{M}_{m}, g_{t}|_{M_m})$ of the same dimension so that $\partial M_{m}$ is a $2$-sided, closed, embedded hypersurface in $\tilde{M}_{m}$. 

By Smale's Transverality Theorem, since the regularity of $\Pi$ is $C^{2}$, we may perturb $g_{t}|_{M_m}$ slightly in the $C^{\infty}$-topology to get another smooth $1$-parameter family $\tilde{g}_{m, t}: [0, 1] \rightarrow \tilde{\Gamma}^{(q)}$ that is transversal to both $\tilde{\Pi}: \widetilde{\mathcal{M}} \rightarrow \tilde{\Gamma}^{(q)}$ and $\tilde{\Pi}: \widetilde{\mathcal{M}}_{\sided} \rightarrow \tilde{\Gamma}^{(q)}$. By Sard's Theorem and the transversality of $\{\tilde{g}_{m, t}\}$, for $t \in [0, 1]$ a.e. in the sense of Lebesgue measure, $\tilde{g}_{m, t}$ is a regular value of $\tilde{\Pi}: \widetilde{\mathcal{M}} \rightarrow \tilde{\Gamma}^{(q)}$, i.e. every embedded minimal hypersurface $\Sigma$ in $(\tilde{M}_{m}, \tilde{g}_{m, t})$ with (possibly empty) free boundary $\partial \Sigma \subset \partial M_{m}$ is non-degenerate. Since $\tilde{\Pi}^{-1}(\{\tilde{g}_{m, t}\})$ is a $1$-dimensional submanifold with boundary of $\widetilde{\mathcal{M}}$ and $\widetilde{\mathcal{M}}_{\sided}$ is a codimension $1$ submanifold, by transversality $\tilde{\Pi}^{-1}(\{\tilde{g}_{m, t}\})$ intersects $\widetilde{\mathcal{M}}_{\sided}$ locally only at finitely many points. Hence, for a.e. $t \in [0, 1]$ in the sense of Lebesgue measure, for any $1$-sided $\Sigma$ above, its $2$-sided double cover $\tilde{\Sigma}$ does not have positive Jacobi fields.   
By restricting to $M_{m}$, $\{g_{t}|_{M_m}\}$ admits a smooth perturbation to $\{\tilde{g}_{m, t}\}$ such that for a.e. $t \in [0, 1]$ in the sense of Lebesgue measure, the following holds: 
\begin{enumerate}[label=(\roman*)]
        \item every almost properly embedded free boundary minimal hypersurface in $(\hat{M}_{m, \tilde{g}_{m, t}}, \tilde{g}_{m, t})$ is non-degenerate,  
        \item for any connected, $1$-sided, almost properly embedded free boundary minimal hypersurface in $(\hat{M}_{m, \tilde{g}_{m, t}}, \tilde{g}_{m, t})$, its $2$-sided double cover has no positive Jacobi field.
\end{enumerate}
Under this assumption on metric, Theorem \ref{DSCW} ensures the existence of a $C^{q}$ compact, multiplicity one, almost properly embedded free boundary minimal hypersurface contained in $(\hat{M}_{m, \tilde{g}_{m, t}}, \tilde{g}_{m, t})$ for a.e. $t \in [0, 1]$.
\\
\\
\indent The subsequent facts are crucial for implementing the global perturbation of a family of metrics. Let $(M^{n + 1}, g)$ be a complete manifold of dimension $3 \leq (n + 1) \leq 7$. For an integer $q \geq 3$ or $q = \infty$, we denote by $\Gamma^{(q)}$ the set of $C^{q}$ metrics on $M$, endowed with the strong $C^{q}$-topology. If we assume $g \in \Int(\mathcal{T}_{\infty})$, then by definition any connected finite volume complete minimal hypersurface in $(M, g)$ is closed. The metric $g \in \Int(\mathcal{T}_{\infty})$ is said to be embedded bumpy (resp. immersed) if every closed embedded (resp. immersed) minimal hypersurfaces in $M$ is non-degenerate. By \cite{White91, White17}, the set of embedded or immersed bumpy metrics is a $C^{q}$-generic subset in $\Int(\mathcal{T}_{\infty})$ for any $q \geq 3$ or $q = \infty$, in the sense of Baire category. 

Fix $q \geq 3$ in this Section and let $g \in \mathcal{F}_{\thin}$ be a $C^{q}$ metric. Let $S \subset (M, g)$ be a connected, closed, embedded, 2-sided minimal hypersurface that is strictly stable. By the Implicit Function Theorem \cite[Theorem 2.1]{White91}, for any metric $\hat{g}$ in a small $C^{q}$-neighborhood of $g$, we can find a unique minimal hypersurface $S_{\hat{g}}$ as a section of the normal bundle of $S$ with small $C^{j, \alpha}$-norm that is also $2$-sided and strictly stable.

Fix a number $\hat{\epsilon} = (\epsilon_1, \epsilon_2, \ldots)$. Recall that if $b_1, b_2, \ldots$ are open balls forming a locally finite covering of $M$, then for any symmetric $2$-tensor $g'$ we will write
\[
||g'||_{C^{q}} < \hat{\epsilon} \hspace{10pt} (\text{or } ||g'||_{C^{q}} < \epsilon \text{ if } \hat{\epsilon} = (\epsilon, \epsilon, \ldots))  
\]
instead of 
\[
\forall j, \hspace{10pt} ||g'|_{b_j}||_{C^{q}} < \epsilon_j,
\]
where the norms $||\cdot||_{C^q}$ are computed with respect to the background metric $g$.

Now, we are ready to state the deformation theorem. 

\begin{thm}\label{DT}
    Let $M^{n + 1}$ be a smooth manifold of dimension $3 \leq (n + 1) \leq 7$ endowed with a smooth metric $g \in \mathcal{F}_{\thin} \cap \Int(\mathcal{T}_{\infty})$. For any integer $l > 0$ and for any connected, closed, embedded, $2$-sided minimal hypersurface $S \subset (M, g)$ that is strictly stable, there is a smooth metric $g' \in \mathcal{F}_{\thin} \cap \Int(\mathcal{T}_{\infty})$ with 
    \[
    ||g' - g||_{C^{l}} \leq \frac{1}{l}, 
    \]
    so that $S$ deforms to a connected, closed, embedded, $2$-sided, strictly stable minimal hypersurface, still denoted by $S$, in $(M, g')$, and there is a non-degenerate, closed, embedded minimal hypersurface $\Sigma \subset (M, g')$ satisfying the following with respect to the metric $g'$: 
    \begin{enumerate}[label=\emph{(\roman*)}]
    \item $\Sigma \cap S = \emptyset$,
    \item $l < ||\Sigma||$,
    \item $\mathbf{F}\left( \frac{[S]}{||S||}, \frac{[\Sigma]}{||\Sigma||} \right) \leq h(||\Sigma||)$ for some function $h$ with $h(||\Sigma||) = 0$ as $||\Sigma|| \rightarrow \infty$. 
    \end{enumerate}
\end{thm}

\begin{proof}
    For a metric $g_{m}$ in a small $C^{3}$-neighborhood of $g|_{M_{m}}$, we denote by $S_{g_{m}}$ the $2$-sided, strictly stable minimal hypersurface coming from $S$. Recall that in Section \ref{twopointtwo}, for each integer $m \geq m_{0}$ we have constructed a compact Riemannian manifold with boundary $(\hat{M}_{m, g_m}, g_{m})$ and showed the existence of a compact, almost properly embedded free boundary minimal hypersurface $\Gamma_{m, p}$ contained in $\hat{M}_{m, g_{m}} \setminus S_{g_{m}}$, satisfying 
    \[\tilde{\omega}_{p}(\hat{M}_{m, g_{m}}, g_m) = \Area_{g_m}(\Gamma_{m, p}), \hspace{10pt} \Index(\Gamma_{m, p}) \leq p.\]
    Let $k$ be a large integer. By constructing certain perturbations $\{\tilde{g}_t\}$ of the given metric $g$, we would like to obtain a closed, non-degenerate minimal hypersurface $\Sigma \subset (M, \tilde{g}_{s})$ of large area, approximating $S$ after renormalization. 
    
    Compared to the proof of \cite[Theorem 3.2]{Song-Zhou21}, the major problem appearing in our case is that as $m \rightarrow \infty$, the sequence of compact, almost properly embedded free boundary minimal hypersurfaces $\{\Gamma_{m, p_{k}}\}$ (with $p_{k}$ fixed) might escape to infinity and no compactness theorem applies to yield a subsequence limit $\Sigma$. To resolve this issue, we first show that as a varifold, $\Gamma_{m, p_{k}}$ has most of its mass concentrated in a small neighborhood of $S$. Then by selecting all connected components of $\Gamma_{m, p_{k}}$ whose support intersect the small neighborhood of $S$, we obtain a sequence of compact, almost properly embedded free boundary minimal hypersurfaces $\{\tilde{\Gamma}_{m, p_{k}}\}$. Borrowing an idea from Song in his proof of Theorem \ref{GODT}, we reason as follows: since the area and index are uniformly bounded, by taking a subsequence limit as $m \rightarrow \infty$, we get a complete, finite volume, embedded minimal hypersurface $\Sigma \subset (M, \tilde{g}_{s})$ intersecting the small neighborhood of $S$. By assumptions on the perturbed metrics $\tilde{g}_{s}$, 
    \begin{enumerate}[label=(\roman*)]
        \item $\tilde{g}_{s} \in \Int(\mathcal{T}_{\infty})$, 
        \item $\tilde{g}_{s}$ is embedded bumpy,  
        \item for any connected, $1$-sided, closed, embedded minimal hypersurface in $(M, \tilde{g}_{s})$, its $2$-sided double cover has no positive Jacobi field, 
    \end{enumerate}
    we deduce that $\Sigma$ must be closed and non-degenerate. The quantitative estimates relating the index and area no longer hold, but other properties hold true, albeit with minor modifications. 

    Now, we construct the aforementioned perturbations of $g$. Consider a smooth non-negative symmetric $2$-tensor $h := \varphi g$ on $M$, where $\varphi: M \rightarrow [0, 1]$ is a smooth cutoff function supported in a $1/(2k)$-neighborhood of $S$, such that $\varphi = 1$ in a tubular neighborhood of $S$. Moreover, we require 
    \[
    |\nabla^{l'} h| \leq C(l, S, g) k^{l'} \hspace{10pt} \text{ for all } 0 \leq l' \leq l. 
    \] 
    
    Set $t_{k} := \exp(-k)$ and consider the $1$-parameter family of smooth metrics on $M$: 
    \[
    g_{t} := g + th, \hspace{10pt} t \in[0, t_{k}].
    \]
    For $k$ large enough, we have 
    \[
    g_{t} \in \mathcal{F}_{\thin} \cap \Int(\mathcal{T}_{\infty}), \hspace{10pt} ||g_{t} - g||_{C^{l}} < \frac{1}{2l}. 
    \]
    
    To ensure the limit minimal hypersurface $\Sigma$ is non-degenerate and multiplicity one, we perturb $g_{t}$ in the strong $C^{\infty}$-topology to get another family $\tilde{g}_{t}$ of $C^{q}$ metrics such that 
    \[
    \tilde{g}_{t} \in \mathcal{F}_{\thin} \cap \Int(\mathcal{T}_{\infty}), \hspace{10pt} ||\tilde{g}_{t} - g||_{C^{l}} < \frac{1}{2l},
    \] 
    and there is a full Lebesgue measure set $\mathcal{A} \subset [0, t_{k}]$ satisfying the following: for any $t \in \mathcal{A}$, $\tilde{g}_{t}$ is embedded bumpy and the $2$-sided double cover of any $1$-sided, closed, embedded minimal hypersurface in $(M, \tilde{g}_{t})$ has no positive Jacobi field. When $M$ is a closed manifold with the dimensional constraint, this perturbation result is guaranteed by the discussion before \cite[Theorem 3.2]{Song-Zhou21}. It remains valid for complete manifolds if we restrict to the subset $\Int(\mathcal{T}_{\infty}) \subset \Gamma^{(q)}$ in which the Manifold Structure Theorem applies.       
    
    For each fixed $m \geq m_0$, by Smale's Transversality Theorem and Sard's Theorem, we may perturb $\tilde{g}_{t}|_{M_m}$ in the $C^{\infty}$-topology into another family $\tilde{g}_{m, t}$ of $C^{q}$ metrics, such that there is a full Lebesgue measure set $\mathcal{A}' \subset \mathcal{A}$ satisfying the following: for any $t \in \mathcal{A}'$ and $m \geq m_{0}$,  
    \begin{enumerate}[label=(\roman*)]
        \item every almost properly embedded free boundary minimal hypersurface in $(\hat{M}_{m, \tilde{g}_{m, t}}, \tilde{g}_{m, t})$ is non-degenerate,  
        \item for any connected, $1$-sided, almost properly embedded free boundary minimal hypersurface in $(\hat{M}_{m, \tilde{g}_{m, t}}, \tilde{g}_{m, t})$, its $2$-sided double cover has no positive Jacobi field.
    \end{enumerate}
    Moreover, we can slightly adjust the perturbations $\{\tilde{g}_{m, t}\}$ so that
    \begin{enumerate}[label=(\roman*)]
        \item $\tilde{g}_{m, t}$ converges smoothly and locally uniformly to $\tilde{g}_{t}$ as $m \rightarrow \infty$,  
        \item for any $k$, all $t \in [0, t_{k}]$ and $m \geq m_{2}$,  
        \begin{align*}\label{are}
        |\Area_{\tilde{g}_{m, t}}(S_{\tilde{g}_{m, t}}) - \Area_{\tilde{g}_{t}}(S_{\tilde{g}_{t}})| \leq t_{k}^{2}. 
        \end{align*} 
        \item for $m$ and $k$ large enough, for all $t \in [0, t_{k}]$, and all $n$-plane $P$ in the Grassmannian of tangent of $(n + 1)$-planes of $\hat{M}_{m, \tilde{g}_{m, t}}$, 
        \[|\Tr_{P, \tilde{g}_{m, t}} \frac{\partial \tilde{g}_{m, t}}{\partial t} - \Tr_{P, \tilde{g}_{t}|_{M_m}} \frac{\partial \tilde{g}_{t}|_{M_m}}{\partial t}| \leq t_{k}.\]
    \end{enumerate}
    
    As $t$ varies, $S_{g_{t}} = S$ remains minimal and satisfies that
    \[
    \frac{d}{dt} \Area_{g_{t}}(S_{g_{t}}) = \frac{n}{2(1 + t)} \Area_{g_{t}}(S_{g_{t}}). 
    \]
    By choosing the perturbation $\{\tilde{g}_{t}\}$ of $\{g_{t}\}$ smaller if necessary, we can ensure that for all $t \in [0, t_{k}]$: 
    \begin{align}\label{ade}
    \begin{split}
        |\Area_{g_{t}}(S_{g_{t}}) - \Area_{\tilde{g}_{t}}(S_{\tilde{g}_{t}})| &\leq t_{k}\\
        \left|\frac{d}{dt} \Area_{\tilde{g}_{t}}(S_{\tilde{g}_{t}}) - \frac{n}{2} \Area_{\tilde{g}_{t}}(S_{\tilde{g}_{t}})\right| &\leq \mathbf{c}_{0} \frac{n t_{k}}{2(1 + t_{k})} \leq \mathbf{c}_{0} t_{k},
    \end{split}
    \end{align}
    where $\mathbf{c}_{0} = \mathbf{c}_{0}(M, g, n, S)$ is a constant. Moreover, there exists a constant $\mathbf{c} = \mathbf{c}(M, g, n, S)$ so that for $m \geq m_{3}$ and $k$ large enough, for all $t \in [0, t_{k}]$, and all $n$-plane $P$ in the Grassmannian of tangent $(n + 1)$-planes of $\hat{M}_{m, \tilde{g}_{m, t}}$, 
    \begin{align}\label{tre}
    \begin{split}
        |\Tr_{P, \tilde{g}_{m, t}} \frac{\partial \tilde{g}_{m, t}}{\partial t} - n \varphi| 
        \leq\hspace{3pt}& |\Tr_{P, \tilde{g}_{m, t}} \frac{\partial \tilde{g}_{m, t}}{\partial t} - \Tr_{P, \tilde{g}_{t}|_{M_m}} \frac{\partial (\tilde{g}_{t}|_{M_m})}{\partial t}|\\
        &+ |\Tr_{P, \tilde{g}_{t}|_{M_m}} \frac{\partial (\tilde{g}_{t}|_{M_m})}{\partial t} - \Tr_{P, g_{t}|_{M_m}} \frac{\partial (g_{t}|_{M_{m}})}{\partial t}|\\
        &+ |\Tr_{P, g_{t}|_{M_m}} \frac{\partial (g_{t}|_{M_{m}})}{\partial t} - n \varphi|\\
        \leq\hspace{3pt}& t_{k} + t_{k} + |\Tr_{P, g_{t}|_{M_m}}(\varphi \cdot (g|_{M_m})) - n \varphi| \leq \mathbf{c}t_{k},\\
        \Tr_{P, \tilde{g}_{m, t}} \frac{\partial \tilde{g}_{m, t}}{\partial t} \leq n + \mathbf{c} t_{k}.
    \end{split}
    \end{align} 

    Now, we proceed to describe a quantitative perturbation argument in the same spirit as \cite{Marques-Neves-Song19} and \cite{Song-Zhou21}, but use the remainder term in Theorem \ref{ALUB}. Recall that in Section \ref{twopointone} we have defined the single-cylindrical width $\tilde{\omega}_{p}(\hat{M}_{m, \tilde{g}_{m, t}}, \tilde{g}_{m, t})$ of the compact manifold with boundary $(\hat{M}_{m, \tilde{g}_{m, t}}, \tilde{g}_{m, t})$. Taking derivatives of $\tilde{\omega}_{p}(\hat{M}_{m, \tilde{g}_{m, t}}, \tilde{g}_{m, t})$ for large enough $p$ will give rise to the desired minimal hypersurfaces. By (\ref{subsc}) of Theorem \ref{ALUB}, for all $p \in \mathbb{N}$ and $t \in [0, t_{k}]$ we have for $m \geq \max\{m_{0}, m_{1}, m_{2}\}$: 
    \begin{align}\label{wle}
    \begin{split}
        \Area_{\tilde{g}_{t}}(S_{\tilde{g}_{t}}) - t_{k}^2 \leq \frac{1}{p} \tilde{\omega}_{p}(\hat{M}_{m, \tilde{g}_{m, t}}, \tilde{g}_{m, t}) &\leq \Area_{\tilde{g}_{t}}(S_{\tilde{g}_{t}}) + t_{k}^2 + \frac{\Area_{\tilde{g}_{t}}(S_{\tilde{g}_{t}}) + t_{k}^2 + f(p) + 2}{p}\\
        &\leq \Area_{\tilde{g}_{t}}(S_{\tilde{g}_{t}}) + t_{k}^{2} + \frac{\Area_{g_{t}}(S_{g_{t}}) + t_{k} + f(p) + 3}{p}\\
        &\leq \Area_{\tilde{g}_{t}}(S_{\tilde{g}_{t}}) + t_{k}^2 + \frac{2^{\frac{n}{2}} \Area_{g}(S) + f(p) + 4}{p}\\
        &\leq \Area_{\tilde{g}_{t}}(S_{\tilde{g}_{t}}) + t_{k}^2 + \frac{g(p)}{p},
    \end{split}
    \end{align}
    where $g(p) = 2^{n/2} \Area_{g}(S) + f(p) + 4$ is a sublinear function of $p$ independent of $t$ and $m$. 

    By Lemma \ref{scwll}, let $\mathcal{A}''$ be the full Lebesgue measure set of times in $\mathcal{A}'$ where the derivatives of $\tilde{\omega}_{p}(\hat{M}_{m, \tilde{g}_{m, t}}, \tilde{g}_{m, t})$ exists for all $m \geq m_{0}$. Given $t_{k} = \exp(-k)$, we can find $p_{k}$ a strictly increasing function of $k$ such that $g(p_{k})/p_{k} = o(t_{k}^2)$ and $p_{k} \rightarrow \infty$ as $k \rightarrow \infty$. By Lemma \ref{DSCW}, at each $s \in \mathcal{A}''$, there is a compact, almost properly embedded free boundary minimal hypersurface $\Gamma_{m, p_{k}} \subset \hat{M}_{m, \tilde{g}_{m, s}} \setminus S_{\tilde{g}_{m, s}}$ such that 
    \begin{align}\label{mmsc}
    \begin{split}
        &\tilde{\omega}_{p_{k}}(\hat{M}_{m, \tilde{g}_{m, s}}, \tilde{g}_{m, s}) = \Area_{\tilde{g}_{m, s}}(\Gamma_{m, p_{k}}), \hspace{10pt} \Index(\Gamma_{m, p_{k}}) \leq p_{k},\\
        &\text{ and } \frac{d}{dt}\bigg|_{t = s} \tilde{\omega}_{p_{k}}(\hat{M}_{m, \tilde{g}_{m, t}}, \tilde{g}_{m, t}) = \int_{\Gamma_{m, p_{k}}} \frac{1}{2} \Tr_{\Gamma_{m, p_{k}}, \tilde{g}_{m, s}} (\frac{\partial \tilde{g}_{m, t}}{\partial t} \bigg|_{t = s}) d\Gamma_{m, p_{k}}.
    \end{split}
    \end{align}
    For each $k \in \mathbb{N}$, $s \in \mathcal{A}'$, and $m \geq \max\{m_0, m_1, m_2, m_3\}$, we have 
    \begin{align}\label{desc}
    \begin{split}
        \frac{d}{dt}\bigg|_{t = s} \frac{1}{p_{k}}\tilde{\omega}_{p_{k}}(\hat{M}_{m, \tilde{g}_{m, t}}, \tilde{g}_{m, t}) &\leq \left(\frac{n}{2} + \mathbf{c}t_{k}\right)\frac{1}{p_k} \tilde{\omega}_{p_{k}}(\hat{M}_{m, \tilde{g}_{m, s}}, \tilde{g}_{m, s})\\
        &\leq \frac{n}{2} \Area_{\tilde{g}_{s}}(S_{\tilde{g}_{s}}) + \mathbf{c}t_{k} \Area_{\tilde{g}_{s}}(S_{\tilde{g}_{s}}) + \left(\frac{n}{2} + \mathbf{c}t_{k}\right) \left(t_{k}^{2} +  \frac{g(p_{k})}{p_{k}}\right)\\
        &\leq \frac{n}{2}\Area_{\tilde{g}_{s}}(S_{\tilde{g}_{s}}) + \mathbf{c}_{1}t_{k} + \mathbf{c}_{1} \frac{g(p_{k})}{p_{k}}\\
        &\leq \frac{d}{dt}\bigg|_{t = s} \Area_{\tilde{g}_{t}}(S_{\tilde{g}_{t}}) + \mathbf{c}_{1}t_{k} + \mathbf{c}_{1} \frac{g(p_{k})}{p_{k}},
    \end{split}
    \end{align}
    where $\mathbf{c}_{1} = \mathbf{c}_{1}(M, g, n, S)$ is a constant that can change from line to line. For $k$ large enough, by (\ref{wle}), (\ref{desc}), and the Fundamental Theorem of Calculus, for each $m \geq \max\{m_0, m_1, m_2, m_3\}$ there exists a closed subset $I_{m} \subset \mathcal{A}''$ with positive Lebesgue measure $r_{k}$ (independent of $m$) such that for every $s_{m} \in I_m$, 
   \begin{equation}\label{ubsc}
       -\frac{1}{k^2} \leq \frac{d}{dt}\bigg|_{t = s_{m}} \frac{1}{p_{k}}\tilde{\omega}_{p_{k}}(\hat{M}_{m, \tilde{g}_{m, t}}, \tilde{g}_{m, t}) - \frac{d}{dt}\bigg|_{t = s_{m}} \Area_{\tilde{g}_{t}}(S_{\tilde{g}_{t}}) \leq \frac{1}{k^2}.  
   \end{equation}
    As $m \rightarrow \infty$, we have a subsequence $s_{m} \rightarrow s$ for some $s \in \mathcal{A}''$ (hence $s \in \mathcal{A}'$). By (\ref{mmsc}), for each $s_{m}$ in the subsequence we find a compact, almost properly embedded free boundary minimal hypersurface $\Gamma_{m, p_{k}} \subset (\hat{M}_{m, \tilde{g}_{m, s_{m}}} \setminus S_{\tilde{g}_{m, s_{m}}}, \tilde{g}_{m, s_{m}})$. This minimal surface $\Gamma_{m, p_{k}}$ is multiplicity one with 
    \begin{align}\label{spsc}
    \begin{split}
        ||\Gamma_{m, p_{k}}|| &= \tilde{\omega}_{p_{k}}(\hat{M}_{m, \tilde{g}_{m, s_{m}}}, \tilde{g}_{m, s_{m}}) > 2l, \hspace{10pt} \Index(\Gamma_{m, p_{k}}) \leq p_{k}, \hspace{10pt} \text{ and }\\
        &||p_{k} \cdot ||\Gamma_{m, p_{k}}||^{-1} - ||S_{\tilde{g}_{m, s_{m}}}||^{-1} || < \frac{2}{l}.
    \end{split}
    \end{align}
    Furthermore, for $m \geq \max\{m_0, m_1, m_2, m_3\}$ we have 
    \begin{align}\label{mlsc}
    \begin{split}
        \frac{1}{k^2} &\geq \left|\int_{\Gamma_{m, p_{k}}} \frac{1}{2} \Tr_{\Gamma_{m, p_{k}}, \tilde{g}_{m, s_{m}}} (\frac{\partial \tilde{g}_{m, t}}{\partial t}\bigg|_{t = s_{m}}) d\Gamma_{m, p_{k}} - \frac{d}{dt}\bigg|_{t = s_{m}} \Area_{\tilde{g}_{t}}(S_{\tilde{g}_{t}})
        \right|\\
        &\geq \left|\frac{\tilde{\omega}_{p_{k}}}{p_{k}} \frac{1}{||\Gamma_{m, p_{k}}||} \int_{\Gamma_{m, p_{k}}} \frac{n \varphi}{2} d\Gamma_{m, p_{k}} - \frac{d}{dt}\bigg|_{t = s_{m}} \Area_{\tilde{g}_{t}}(S_{\tilde{g}_{t}})\right| - \mathbf{c}_{2}t_{k}\\
        &\geq \left|\frac{\tilde{\omega}_{p_{k}}}{p_{k}} \frac{1}{||\Gamma_{m, p_{k}}||} \int_{\Gamma_{m, p_{k}}} \frac{n \varphi}{2} d\Gamma_{m, p_{k}} - \frac{n}{2} \Area_{\tilde{g}_{s_m}}(S_{\tilde{g}_{s_m}})\right| - \mathbf{c}_{2}t_{k}\\
        &\geq \frac{n}{2} \Area_{\tilde{g}_{s_m}}(S_{\tilde{g}_{s_m}}) \left|\frac{1}{||\Gamma_{m, p_{k}}||} \int_{\Gamma_{m, p_{k}}} \varphi d\Gamma_{m, p_{k}} - 1\right| - \mathbf{c}_{2}t_{k} - \mathbf{c}_{2} \frac{g(p_{k})}{p_{k}}, 
    \end{split}
    \end{align}
    where $\mathbf{c}_{2} = \mathbf{c}_{2}(M, g, n, S)$ is a constant that can change from line to line. From that we get for $k$ large enough: 
 \begin{equation}\label{klvd}
    \left|\frac{1}{||\Gamma_{m, p_{k}}||} \int_{\Gamma_{m, p_{k}}} \varphi d\Gamma_{m, p_{k}} - 1\right| < \frac{1}{k + 1}. 
\end{equation}
    By the choice of $\varphi$, this further implies that the varifold $\Gamma_{m, p_{k}}$ has most of its mass concentrated in a $1/k$-neighborhood of $S_{\tilde{g}_{s}}$: 
\begin{equation}\label{mcvd}
    \frac{||\Gamma_{m, p_{k}} \setminus N_{1/k}(S_{\tilde{g}_{s}})||}{||\Gamma_{m, p_{k}}||} \leq \frac{1}{k + 1}. 
\end{equation}
    Pick all connected components of $\Gamma_{m, p_{k}}$ whose support intersect the $1/(2k)$-neighborhood of $S$. This gives a varifold $\tilde{\Gamma}_{m, p_{k}}$ which has most of its mass concentrated in a $1/k$-neighborhood of $S_{\tilde{g}_{s}}$: 
\begin{equation}\label{scvd}
    \frac{||\tilde{\Gamma}_{m, p_{k}} \setminus N_{1/k}(S_{\tilde{g}_{s}})||}{||\tilde{\Gamma}_{m, p_{k}}||} \leq \frac{1}{k}. 
\end{equation}

\noindent By fixing $k$ and varying $m$, we obtain a sequence of compact, almost properly embedded free boundary minimal hypersurfaces $\{\tilde{\Gamma}_{m, p_{k}}\}$. Moreover, they admit a uniform bound on the area and index: 
    \begin{align*}
    \Area_{\tilde{g}_{m, s_{m}}}(\tilde{\Gamma}_{m, p_{k}}) \leq \Area_{\tilde{g}_{m, s_{m}}}(\Gamma_{m, p_{k}}) &= \tilde{\omega}_{p_{k}}(\hat{M}_{m, \tilde{g}_{m, s_{m}}}, \tilde{g}_{m, s_{m}}) \leq p_{k} \cdot 2^{\frac{n}{2}} \Area_{g}(S) + 2 p_{k} t_{k} + g(p_{k}), \\
    \Index(\tilde{\Gamma}_{m, p_{k}}) &\leq \Index(\Gamma_{m, p_{k}}) \leq p_{k}.
    \end{align*}

    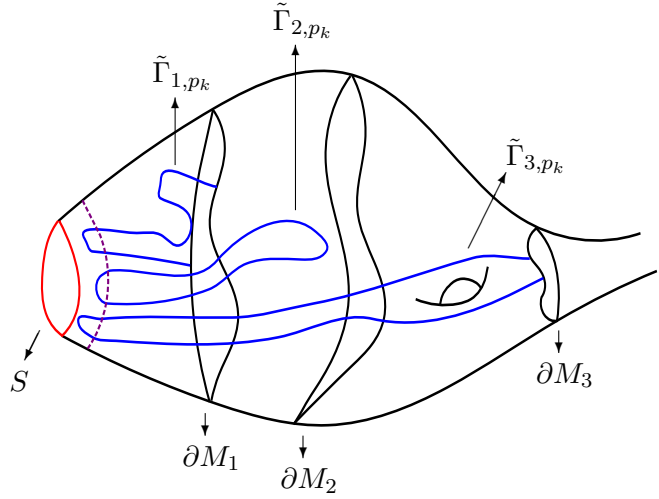
\begin{figure}[h]
\setlength{\unitlength}{1cm}% set the unit length
\centering
\begin{picture}(11, 5)(0,0)
\put (2,0.5) {\includegraphics[scale = 0.6]{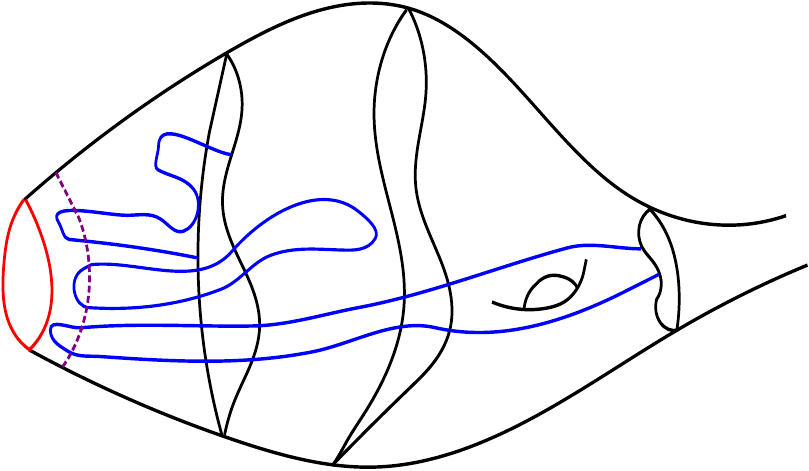}}
\put(2.0, 1.8){\vector(-1, -2){0.2}}
\put(1.6, 1.0){$S$}
\put(4.2, 0.7){\vector(0, -1){0.3}}
\put(3.9, 0){$\partial M_1$}
\put(5.5, 0.4){\vector(0, -1){0.3}}
\put(5.2, -0.3){$\partial M_2$}
\put(8.9, 1.8){\vector(0, -1){0.3}}
\put(8.6, 1.1){$\partial M_3$}
\put(3.8, 4){\vector(0, 1){0.9}}
\put(3.5, 5.1){$\tilde{\Gamma}_{1, p_k}$} 
\put(5.4, 3.4){\vector(0, 1){2.2}}
\put(5.1, 5.8){$\tilde{\Gamma}_{2, p_k}$} 
\put(7.7, 2.9){\vector(1, 2){0.5}}
\put(8.2, 4){$\tilde{\Gamma}_{3, p_k}$} 
\end{picture}
\caption{A schematic illustration of a sequence of minimal hypersurfaces $\{\tilde{\Gamma}_{m, p_{k}}\}$.} 
\label{fig:aa}
\end{figure}

    Given that $\tilde{g}_{s} \in \Int(\mathcal{T}_{\infty})$, there exists a large integer $m''$ such that for each $m \geq m''$, $\tilde{\Gamma}_{m, p_{k}}$ is contained in the interior of $M_{m'}$ and hence closed. Since $\tilde{g}_{m, s_{m}} \rightarrow \tilde{g}_{s}$ smoothly and locally uniformly, by taking a subsequence limit as $m \rightarrow \infty$ we get a closed, embedded minimal hypersurface $\Sigma$ in $(M, \tilde{g}_{s})$ possibly with multiplicities. By assumption, $\tilde{g}_{s}$ is embedded bumpy and the $2$-sided double cover of any $1$-sided, closed, embedded minimal hypersurface in $(M, \tilde{g}_{s})$ has no positive Jacobi field. A standard Jacobi field argument implies that $\Sigma$ is non-degenerate and multiplicity one. From (\ref{mmsc}), (\ref{scvd}), and the following facts 
    \begin{align}\label{mlae}
    \begin{split}
        ||\tilde{\Gamma}_{m, p_{k}}|| \geq \frac{k}{k + 1} \cdot ||\Gamma_{m, p_{k}}|| &\geq \frac{2k}{k + 1} l > \frac{5l}{4} \text{ and }\\
        \lim_{m \rightarrow \infty}|| \tilde{\Gamma}_{m, p_{k}}|| &= ||\Sigma||,
    \end{split}
    \end{align}
    we deduce that for $k$ large enough, 
    \begin{equation}\label{klae}
        ||\Sigma|| > l, \hspace{10pt} |p_{k} \cdot ||\Sigma||^{-1} - ||S_{\tilde{g}_{s}}||^{-1}| < \frac{1}{l}, \hspace{10pt} \text{ and } \frac{||\Sigma \setminus N_{1/k}(S_{\tilde{g}_{s}})||}{||\Sigma||} \leq \frac{1}{k}. 
    \end{equation}

\noindent By Section \ref{twopointfour} and the quantitative constancy theorem, (\ref{klae}) implies that for some constant $C = C(M, g, n, S)$: 
    \begin{equation}\label{qctvd}
        \mathbf{F}\left(\frac{[S_{\tilde{g}_{s}}]}{||S_{\tilde{g}_{s}}||}, \frac{[\Sigma]}{||\Sigma||} \right) \leq 
        \frac{C}{\sqrt{k}}. 
    \end{equation}
    Define $\tilde{h}$ a function of $p_{k}$ by $\tilde{h}(p_{k}) = C/\sqrt{k}$. By (\ref{klae}), we get for large $k$:
    \begin{equation}\label{klhf}
        \frac{C}{\sqrt{k}} = \tilde{h}(p_{k}) = h(||\Sigma||),
    \end{equation}
    where $h$ is a function of $||\Sigma||$. As $k \rightarrow \infty$, we have $p_{k} \rightarrow \infty$, $||\Sigma|| \rightarrow \infty$, and $h(||\Sigma||) \rightarrow 0$. Hence for a large $k$, we perturb the $C^{q}$ metric $\tilde{g}_s$ to a smooth metric $g' \in \mathcal{F}_{\thin} \cap \Int(\mathcal{T}_{\infty})$ with $||g' - g||_{C^{l}} \leq 1/l$, and the proof is completed by (\ref{klae}), (\ref{qctvd}), (\ref{klhf}), and non-degeneracy of $\Sigma$.
\end{proof}

\section{Generic Scarring for minimal hypersurfaces along stable hypersurfaces} \label{sectionfour} 
\begin{proof}[Proof of Theorem \ref{MAIN}]
    Let $U \subset M$ be a compact domain with no minimal boundary component and let $A > 0$. We say that a smooth metric $g \in \mathcal{F}_{\thin} \cap \Int(\mathcal{T}_{\infty})$ on $M$  satisfies property $(P_{A, U})$ if: \\
    \\
    $(P_{A, U})$ any connected, closed, embedded, stable minimal hypersurface $S \subset (U, g)$ with area at most $A$ has a non-degenerate $2$-sided double cover. If moreover $S$ is $2$-sided, there is a closed, embedded, non-degenerate minimal hypersurface $\Sigma$ satisfying 
    \begin{enumerate}[label=(\roman*)]
    \item $\Sigma \cap S = \emptyset$,
    \item $A < ||\Sigma||$,
    \item $\mathbf{F}\left( \frac{[S]}{||S||}, \frac{[\Sigma]}{||\Sigma||} \right) \leq h(||\Sigma||)$.
    \end{enumerate}
    
    The set of metrics satisfying $(P_{A, U})$ is denoted by $\mathcal{M}_{A, U}$. We want to show that for all $A > 0$ and $U \subset M$ a compact domain with no minimal boundary component, $\mathcal{M}_{A, U}$ is open and dense in $\mathcal{F}_{\thin} \cap \Int(\mathcal{T}_{\infty})$. Then for any compact exhaustion $\{U_{m}\}$ of $M$ (each $U_{m}$ is a compact domain with no minimal boundary component), $\bigcap_{m \in \mathbb{N}} \bigcap_{A \in \mathbb{N}} \mathcal{M}_{A, U_{m}}$ would be a $C^{\infty}$-generic subset of $\mathcal{F}_{\thin} \cap \Int(\mathcal{T}_{\infty})$ in the sense of Baire and any metric in this family would satisfy the main theorem.    

    Fix $A > 0$ and $U \subset M$ a compact domain with no minimal boundary component. Given a smooth metric $g \in \mathcal{F}_{\thin} \cap \Int(\mathcal{T}_{\infty})$ on $M$ satisfying the property $(P_{A, U})$, Sharp's compactness result and a standard Jacobi field argument ensure that there exists a small $\epsilon_{0} = \epsilon_{0}(U, g)$ such that no connected, closed, embedded, stable minimal hypersurface in $(U, g)$ has area lying inside $(A, A + \epsilon_0)$, and there are only finitely many such stable minimal hypersurfaces in $(U, g)$ with area at most $A + \epsilon_0$. Openness of $\mathcal{M}_{A, U}$ then follows from the Implicit Function Theorem.

    To prove denseness of $\mathcal{M}_{A, U}$, consider an immersed bumpy metric $g \in \mathcal{F}_{\thin} \cap \Int(\mathcal{T}_{\infty})$ on $M$. Recall that in Section \ref{twopointthree}, we have confirmed that the set of immersed bumpy metrics is dense in $\mathcal{F}_{\thin} \cap \Int(\mathcal{T}_{\infty})$. By generic finiteness, there exists a small $\epsilon_{0} = \epsilon_{0}(U, g)$ such that no connected, closed, embedded, stable minimal hypersurface in $(U, g)$ has area lying inside $(A - \epsilon_{0}, A + \epsilon_0)$, and there are only finitely many such stable minimal hypersurfaces in $(U, g)$ with area at most $2A$. We denote those connected, closed, embedded, stable minimal hypersurfaces with area less than $A$ which are $2$-sided by $\{S_{1}, \ldots, S_{J}\}$, and those which are $1$-sided by $\{T_1, \ldots, T_{k}\}$. Arguing as in the proof of \cite[Theorem 0.1]{Song-Zhou21}, we apply the deformation result Theorem \ref{DT} to each $S_{j}$ successively and eventually get a metric $g_{J}$ with $||g - g_{J}||_{C^{l}} \leq \epsilon$ and non-degenerate minimal hypersurfaces $\Sigma_{1}, \ldots, \Sigma_{J}$. The deformation can be chosen small enough at each step so that the only connected, closed, embedded, stable minimal hypersurfaces with area less than $A$ in $(U, g_{J})$ are $S_{1}, \ldots, S_{J}$ and the minimal hypersurfaces coming from $T_{1}, \ldots, T_{K}$. Thus $g_{J} \in \mathcal{M}_{A, U}$ and we finish the whole proof.  
\end{proof}

\section{Appendix: Proof of Theorem \ref{qctc}}

\begin{proof}
    We closely follow the proof of \cite[Theorem 2.1]{Song-Zhou21}. Thus we describe the argument in a relatively informal way, relying on \cite{Song-Zhou21} for full details.

    Assume $||V||(M) = 1$ and write $d\mu_{S} = dS/||S||$. Our goal is to estimate the quantity 
    \[
    \left| \int_{G_{n}(M)} f(x, P) dV(x, P) - \int_{S} f(p, T_{p} S) d\mu_{S}(p) \right|
    \]
    for any Lipschitz function $f: G_{n}(M) \rightarrow \mathbb{R}$ with $||f||_{\infty} < \infty$ and $\Lip(f) \leq 1$. 
    
    To begin with, we show that $V$ is supported over $n$-planes that are ``almost parallel'' to $\{S_{t}\}$, i.e. there exists a constant $C = C(g, S)$ such that for any $0 < \delta < \epsilon$, 
    \begin{equation}\label{ap}
    \int_{G_{n}(N_{\delta}(S))} (1 - (\nu(x) \cdot \nu_{P})^2) dV(x, P) < C(g, S) \delta. 
    \end{equation} 
    To see this, we first extend $t: N_{\epsilon}(S) \rightarrow (-\epsilon, \epsilon)$ to a compactly supported function $t: M \rightarrow \mathbb{R}$ with $\sup_{b_{k}} ||t||_{C^{2}} \leq C(g, S)$. By plugging the vector field $X = t \nabla t$ into the first variation formula for $V$ and using (\ref{qct}), we obtain the estimate (\ref{ap}). Then the Cauchy-Schwartz inequality leads to 
    \begin{align}\label{csi}
    \int_{G_{n}(N_{\delta}(S))} \dist_{g}(P, T_{x} S_{x}) dV(x, P) &\leq \int_{G_{n}(N_{\delta}(S))} (1 - (\nu(x) \cdot \nu_{P})^2) dV(x, P).
    \end{align} 
    By (\ref{qct}) and (\ref{csi}), we have 
    \begin{align*}
        &\left| \int_{G_{n}(M)} f(x, P) dV(x, P) - \int_{S} f(p, T_{p} S) d\mu_{S}(p) \right|\\ \leq\hspace{3pt}& C(g, S) \sqrt{\delta} + \left| \int_{N_{\delta}(S)} f(\pi(x), T_{\pi(x)} S) d\mu_{V}(x) - \int_{S} f(p, T_{p} S) d\mu_{S}(p) \right|. 
    \end{align*}

    It remains to estimate the second term. Let $f_{S}(p) = f(p, T_{p} S)$ be a function on $S$ and denote its average over $S$ by $f_{\av} = \int_{S} f_{S} d\mu_{S}$. Consider the PDE on $S$ under the assumption that $\int_{S} \varphi d\mu_{S} = 0$: 
    \begin{equation}\label{pdes}
    f_{S} = f_{\av} + \Delta_{S} \varphi. 
    \end{equation} 
    The standard elliptic estimates give $||\varphi||_{C^{2, \alpha}(S)} \leq C(g, S)$. We extend $\varphi: S \rightarrow \mathbb{R}$ to a compactly supported function $\varphi: M \rightarrow \mathbb{R}$ with $\sup_{b_{k}} ||\varphi||_{C^{2}} \leq C(g, S)$. Using the following Hessian estimate
    \begin{equation}\label{hes}
        |\Tr_{S_{x}} (\Hess \varphi)(x) - \Tr_{S_{x}} (\Hess \varphi)(\pi(x))| \leq C(g, S) \delta
    \end{equation}
    together with (\ref{qct}), (\ref{csi}), (\ref{pdes}), and (\ref{hes}), we derive the estimate 
    \begin{align*}
        &\left| \int_{N_{\delta}(S)} f(\pi(x), T_{\pi(x)} S) d\mu_{V}(x) - \int_{S} f(p, T_{p} S) d\mu_{S}(p) \right| \\
        \leq\hspace{3pt}& C(g, S) \sqrt{\delta} + \left| \int_{G_{n}(M)} \divergence_{P}(\nabla \varphi)(x) dV(x, P)\right|.  
    \end{align*}
    Since $V$ is stationary, the second term vanishes and we finish the proof. Clearly, the constant $C = C(g, S)$ can be chosen to be uniformly bounded in a $C^{3}$-neighborhood of $(g, S)$. 
\end{proof}

\bibliographystyle{amsplain}
\bibliography{refs}
\nocite{Allard72, Almgren62, Almgren65, Chodosh-Ketover-Maximo15, Gromov88, Guth09, LiYang20, Pitts81, Schoen-Simon81, Sharp17, Simon83, Sma68, Zhou-Zhu20}
\end{document}